\documentclass[12pt,leqno,a4paper]{amsart}
\usepackage{amssymb,enumerate}
\newcommand{\FF}{{\mathbb{F}}}
\newcommand{\PP}{{\mathbb{P}}}

\newcommand{\fA}{{\mathfrak{A}}}
\newcommand{\fS}{{\mathfrak{S}}}

\newcommand{\bB}{{\mathbf{B}}}
\newcommand{\bG}{{\mathbf{G}}}
\newcommand{\bT}{{\mathbf{T}}}
\newcommand{\bU}{{\mathbf{U}}}

\newcommand{\cE}{{\mathcal{E}}}

\newcommand{\ad}{{\operatorname{ad}}}

\newcommand{\Aut}{{\operatorname{Aut}}}
\newcommand{\Irr}{{\operatorname{Irr}}}
\newcommand{\SC}{{\operatorname{sc}}}

\newcommand{\tr}{{\operatorname{tr}}}
\newcommand{\rk}{{\operatorname{rk}}}

\newcommand{\PGL}{{\operatorname{PGL}}}
\newcommand{\PSL}{{\operatorname{L}}}
\newcommand{\SL}{{\operatorname{SL}}}
\newcommand{\GL}{{\operatorname{GL}}}
\newcommand{\PSU}{{\operatorname{U}}}
\newcommand{\GU}{{\operatorname{GU}}}
\newcommand{\SU}{{\operatorname{SU}}}
\newcommand{\Sp}{{\operatorname{Sp}}}
\newcommand{\Spin}{{\operatorname{Spin}}}
\newcommand{\PSp}{{\operatorname{S}}}
\newcommand{\OO}{{\operatorname{O}}}
\newcommand{\SO}{{\operatorname{SO}}}

\newcommand{\half}{{\frac{1}{2}}}

\newcommand{\tw}[1]{{}^#1\!}

\def\skipa{\vspace{-1.5mm} & \vspace{-1.5mm} & \vspace{-1.5mm}\\}
\renewcommand{\atop}[2]{\genfrac{}{}{0pt}{}{#1}{#2}}

\let\eps=\epsilon
\let\co=\colon
\let\opl=\oplus

\newtheorem{thm}{Theorem}[section]
\newtheorem{lem}[thm]{Lemma}
\newtheorem{cor}[thm]{Corollary}
\newtheorem{prop}[thm]{Proposition}

\theoremstyle{definition}
\newtheorem{defn}[thm]{Definition}

\theoremstyle{remark}
\newtheorem{rem}[thm]{Remark}

\begin{document}

\title{Simple groups admit Beauville structures}

\date{\today}

\author{Robert Guralnick}
\address{3620 S. Vermont Ave, Department of Mathematics, University of
Southern California, Los Angeles, CA 90089-2532, USA.}
\makeatletter
\email{guralnic@usc.edu}
\makeatother
\author{Gunter Malle}
\address{FB Mathematik, TU Kaiserslautern,
Postfach 3049, 67653 Kaisers\-lautern, Germany.}
\makeatletter
\email{malle@mathematik.uni-kl.de}
\makeatother

\thanks{The first author was partially supported by  NSF grants
  DMS 0653873 and 1001962.}

\begin{abstract}
We answer a conjecture of Bauer, Catanese and Grunewald showing that
all finite simple groups other than the alternating group of degree $5$ admit
unmixed Beauville structures. We also consider an analog of the result
for simple algebraic groups which depends on some upper bounds for
character values of regular semisimple elements in finite groups of Lie type.
Finally, we prove that any finite simple group contains two conjugacy classes
$C, D$ such that any pair of elements in $C \times D$ generates the group.
\end{abstract}

\dedicatory{Dedicated to the memory of Fritz Grunewald}

\maketitle

%%\pagestyle{myheadings}
%%\markboth{for personal use only}{}
%%\markboth{}{}

%%%%%%%%%%%%%%%%%%%%%%%%%%%%%%%%%%%%%%%%%%%%% 
\section{Introduction} \label{sec:intro}

A Beauville surface is a compact complex surface $S$ which is rigid (i.e.,
it has no non-trivial  deformations) and satisfies
$S =  (X \times Y )/G $ where  $X$ and $Y$ are curves of genus at least $2$
and $G$ is a finite group acting freely on $X \times Y$.   See \cite{BCG06}
for more background on the history and importance of  Beauville surfaces.

A finite group $G$ is said to admit an \emph{unmixed Beauville structure} if
there exist two pairs of generators $(x_i, y_i)$, $i =1,2$, for $G$ such
that $\Sigma(x_1, y_1) \cap \Sigma(x_2,y_2) = \{1\}$, where for $x,y \in G$
we set
$$\Sigma(x,y)=\bigcup_{i\ge0,\,g\in G}\{gx^ig^{-1},gy^ig^{-1},g(xy)^ig^{-1} \}.
$$
In particular, if there are two generating pairs $(x_i,y_i)$ such that
the orders of $x_1, y_1$ and $x_1y_1$ are relatively coprime to those of
$x_2,y_2$ and $x_2y_2$, then $G$ admits an unmixed Beauville structure.
By the Riemann existence theorem, each generating pair $(x_i, y_i)$ of $G$
gives rise to a Galois action of $G$ on a curve $X_i$ such that
$X_i\rightarrow X_i/G \cong \PP^1$ is branched at $3$ points. The condition
that $\Sigma(x_1,y_1)\cap\Sigma(x_2,y_2)=\{1\}$
is precisely the condition that the action of $G$ on $X_1\times X_2$ is free. 
 
%%% real structure (at least orders are distinct) does there exist an inner
%%%  automorphism inverting x_i, y_i, we are not really addressing this 

Our first main result completes a program announced and started in \cite{GM10}
and answers a conjecture of Bauer--Catanese--Grunewald \cite{BCG05,BCG06}
regarding Beauville structures. See also \cite{FJ} for some low rank cases.

\begin{thm}   \label{thm:beau}
 Let $G$ be a finite non-abelian simple group other than $\fA_5$. Then $G$
 admits an unmixed Beauville structure.
\end{thm}

The proof will be given in Sections~\ref{sec:exc}--\ref{sec:altspor}.

See Garion--Larsen--Lubotzky \cite{GLL} for a proof that
this holds for all sufficiently large simple groups of Lie type, and
Fuertes--Gonz\'alez-Diez \cite[Thm.~1]{FG} for the alternating groups.
Our proof is independent of the result in \cite{GLL}. The statement for
alternating groups also follows easily by \cite{Be72}.
\vskip 1pc

In the proof, the well-known character formula for structure constants in
finite groups will be essential:
Let $G$ be a finite group, $C_1,C_2,C_3$ conjugacy classes of $G$. Then for
fixed $x\in C_1$ the number of pairs
$$n(C_1,C_2,C_3):=|\{(y,z)\in C_2\times C_3\mid xyz=1\}|$$
in $G$ is given by the character formula
$$n(C_1,C_2,C_3)=\frac{|C_2|\cdot|C_3|}{|G|}
  \sum_{\chi\in\Irr(G)}\frac{\chi(C_1)\chi(C_2)\chi(C_3)}{\chi(1)},$$
where the sum ranges over the complex irreducible characters of $G$ and
$\chi(C_i)$ denotes the value of $\chi$ on elements of $C_i$.
%%(see for example \cite[Thm.~I.5.8]{MM}).

If $G$ is simple of Lie type, the first two conjugacy classes contain regular
semisimple elements and the third class non-identity semisimple elements,
then the structure constant is always non-zero by the nice result of
Gow \cite[Thm.~2]{Gow}.

We will also use several results about maximal subgroups containing special
elements (mostly based on Guralnick--Penttila--Praeger--Saxl \cite{GPPS}).
We also obtain some new results that may be of independent interest about
maximal subgroups of the exceptional groups, see Theorem~\ref{thm:excelt}. 
In general, the idea of the proof is fairly simple. We find three conjugacy
classes $C_i$ of our simple group such that there are no
maximal subgroups intersecting all three classes (or at least very few). 

These ideas allow us to prove some related results.  The first is:

\begin{thm}   \label{thm:alggroup triples}
 If $G$ is a simply connected simple algebraic group  of rank $r > 1$ over
 an algebraically closed field , and $C_i$, $1 \le i \le 3$, are conjugacy
 classes of regular semisimple elements of $G$, then the variety
 $\{(x_1, x_2, x_3)\mid x_i\in C_i,\,x_1x_2x_3=1\}$  is irreducible of
 dimension $2 \dim G  -3r$.
\end{thm}

We prove the result over the algebraic closure of a finite field by using our
results on the finite groups of Lie type.  From this, the theorem follows
by a simple argument (as pointed out to us by Michael Larsen). See
Theorem~\ref{thm:dimtriples} and Remark~\ref{rem:larsen}.
For  the exceptional groups, we prove the analogous
result with $C_3$ an almost arbitrary  class. See
Theorem \ref{thm:moretriples} for the precise statement. 

The ideas used in the proof also allow us to show (see
Section~\ref{sec:genclass}):

\begin{thm}   \label{thm:gen classes}
 Let $G$ be a finite simple group. There exist conjugacy classes $C$ and $D$
 of $G$ such that $G = \langle c, d \rangle$ for any $c \in C$ and $d\in D$.
\end{thm}

In \cite[Thm. A]{DHP}, finite solvable groups were characterized by the
property that for any pair of conjugacy classes $C,D$, there exist
$(c,d) \in C \times D$ with $\langle c, d \rangle$ solvable.
Using a variation of the  reduction to almost simple groups in \cite{DHP} and 
a slight generalization of the previous result characterizes any family of
finite groups closed under subgroups, quotients and extensions in a similar
fashion. More precisely, one has the following result (see also
\cite[Thm.~C]{DGHP}):

\begin{cor} \label{classes}
 Let $X$ be a family of finite groups closed under subgroups, quotients and
 extensions. A finite group $G$ belongs to $X$ if and only if for every
 $x,y \in G$, $\langle x, y^g \rangle \in X$ for some $g \in G$.
\end{cor}

The paper is organized as follows. In the next three sections, 
we prove Theorem \ref{thm:beau} for exceptional groups, classical
groups and sporadic groups. 

In Section \ref{sec:bound}, we prove Theorem \ref{thm:bound-dis} which gives
an upper bound for the absolute value of character values on semisimple
elements in groups of Lie type (including disconnected groups --- see
\cite[Thm.~3]{GLL} for a different proof with a larger bound). We then
use this result to study the structure of the variety of triples of elements
in three semisimple regular classes with product 1.

In the final section, we prove a slightly more general version of 
Theorem~\ref{thm:gen classes} that allows one to obtain Corollary \ref{classes}.
\\

\noindent
{\bf Acknowledgement}.  We wish to thank M. Larsen and T. Springer
for various remarks about the results of this paper.  In particular, their
comments allowed us to extend the proof of Theorem \ref{thm:alggroup triples}
to  all semisimple regular classes and to characteristic $0$ and to prove
the irreducibility of the variety. 
\vskip 1pc

\noindent
{\bf Remark.}
 After this paper was completed and posted, we were informed by Fairbairn,
 Magaard and Parker that they also proved Theorem~\ref{thm:beau} using similar
 methods (see \cite{FMP}). Also, Kantor, Lubotzky and Shalev have just sent
 us a preprint \cite{KLS} including a proof of Theorem  \ref{thm:gen classes}.

%%%%%%%%%%%%%%%%%%%%%%%%%%%%%%%%%%%%%%%%%%%%%%%%%%%%%%%%%%%%%%%%%%%%%%%%%
\section{Exceptional groups}   \label{sec:exc}

Here, we show Theorem~\ref{thm:beau} for the exceptional groups of Lie type;
the Tits group $\tw2F_4(2)'$ will be considered in Proposition~\ref{prop:spor}.
We also exclude the solvable group $\tw2B_2(2)$ and the non-simple groups with
classical socle $G_2(2)\cong\PSU_3(3).2$ and $\tw2G_2(3)\cong\PSL_2(8).3$
throughout this section.

First we prove a result on overgroups of certain maximal tori, which may be
of independent interest. In the case of $E_7(q)$ this is due to Weigel
\cite[4(i)]{We92}.

\begin{thm}   \label{thm:excelt}
 Let $G$ be a simple exceptional group of Lie type different from $\tw3D_4(q)$.
 Then there exists a cyclic subgroup $T\le G$ such that $|T|$ and the
  maximal overgroups $M\ge T$ in $G$ are as given in Table~\ref{tab:exctorus}.
\end{thm}

\begin{table}[htbp] 
  \caption{Maximal overgroups of cyclic subgroups in exceptional groups}
  \label{tab:exctorus}
\[\begin{array}{|l||c|l|l|} 
\hline
 G& |T|& M\ge T& \text{further maximal overgroups}\\
\skipa \hline \hline
 \tw2B_2(q^2),\ q^2\ge8& \Phi_8''& T\co4& -\\ \hline
 ^2G_2(q^2),\ q^2\ge27& \Phi_{12}''& T\co6& -\\ \hline
 G_2(q),\ 3|q-\eps& q^2+\eps q+1& \SL_3^\eps(q).2& q=4:\ J_2,G_2(2),\PSL_2(13)\\ \hline
 G_2(q),\ 3|q& q^2-q+1& \SU_3(q).2\ (2\times)& q=3:\ ^2G_2(3),2^3.\PSL_3(2),\PSL_2(13)\\ \hline
%%   \tw3D_4(q)& q^2+q+1& N_G(T),G_2(q)& ????\\ \hline
 \tw2F_4(q^2),\ q^2\ge8& \Phi_{24}''& T\co12& -\\ \hline
 F_4(q),\ 2{\not|}q& \Phi_8& \Spin_9(q)& -\\ \hline
 F_4(q),\ 2|q& \Phi_8& \SO_9(q)\ (2\times)& -\\ \hline
       E_6(q)& \Phi_3\Phi_{12}/d& \Phi_3.\tw3D_4(q).3/d& -\\ \hline
   \tw2E_6(q)& \Phi_6\Phi_{12}/d& \Phi_6.\tw3D_4(q).3/d& q=2:\ F_4(2),Fi_{22}\ (\text{3 each})\\ \hline
       E_7(q)& \Phi_1\Phi_9/d& P,P',L.2& -\\ \hline
%%       E_7(q)& \Phi_2\Phi_{14}/d& U_8(q),\PSL_2(q^7).7d& ??\\ \hline
       E_8(q)& \Phi_{15}& T\co30& -\\ \hline
\end{array}\]
\end{table}

In the table, $\Phi_i$ denotes the $i$th cyclotomic polynomial evaluated at~$q$,
$\Phi_8''=q^2-\sqrt{2}q+1$, $\Phi_{12}''=q^2-\sqrt{3}q+1$,
$\Phi_{24}''=q^4-\sqrt{2}q^3+q^2-\sqrt{2}q+1$. For $G_2(q)$, $\eps\in\{\pm1\}$.
In $E_7(q)$, $P,P'$ denote two maximal $E_6$-parabolic subgroups, and $L$ their
common Levi factor; $d=\gcd(3,q-1)$ for $E_6(q)$, $d=\gcd(3,q+1)$ for
$\tw2E_6(q)$, $d=\gcd(2,q-1)$ for $E_7(q)$.

\begin{proof}
The existence of maximal tori of the given orders follows from general
theory, see \cite[Prop.~25.3]{MT10} for example.
The maximal subgroups of the exceptional groups of Lie type of small rank are
known explicitly, see \cite{Co81,KlG2,Ma91}. From those lists, it
is straightforward to check the first five lines of the table.   \par
For $E_7(q)$ the claim is in \cite[4(i)]{We92}. For $F_4(2),\tw2E_6(2)$ and
$E_6(2)$ the maximal subgroups are listed in \cite{Atl}. For the remaining
exceptional groups of large rank, we use the results of Liebeck--Seitz
\cite{LS03}. Let $M$ be a maximal subgroup of $G$ containing $T$. By
\cite[Thm.~8]{LS03}, using Zsigmondy prime divisors of $|T|$, one finds that
either $M$ is reductive of maximal rank as given in the table, or almost
simple. In the latter case, by \cite[Table~2]{LS03} the socle $S$ of $M$ is
of Lie type in the same characteristic as $G$. By
\cite[Thm.~8(VI)]{LS03} the untwisted rank of $S$ is at most half the rank
of $G$, and either $S$ is defined over a field of order at most~9, or it
is $^\eps\PSL_3(16)$ or of rank~1. It ensues that the only
possibilities are $S=\PSp_4(9)$ or $\PSL_2(81)$ inside $F_4(3)$.
\par
So assume $G=F_4(3)$. The torus of order $q^4+1$ is contained in the subsystem
subgroup $B_4(q)=\Spin_9(q)$, for which the 25-dimensional $G$-module $V$ has
two composition factors: once the 16-dimensional spin-module and once the
natural module for $\SO_9(q)$. Any regular element $x$ of order $q^4+1$ is
conjugate to 8 of its powers, so it fixes a 1-dimensional subspace of $V$.
First assume that $S=\PSp_4(9)$ embeds into $G$.
By the theory of irreducible $\bar\FF_3\Sp_4$-modules, the irreducible
$\FF_3 S$-modules have dimension~1,10,16,20,25 or bigger. Since $S$ can
have at most one fixed point, this shows that $V|_S$ is irreducible (in fact,
absolutely irreducible). But by \cite[Cor.~2 (Table 1.3)]{LS04} there is no
such irreducible embedding. \par
So now consider $S=\PSL_2(81)$. Again, it's
easily seen that the irreducible $\FF_3 S$-modules have dimension~1,8,12,16 or
bigger than~25. Here, the 12-dimensional module cannot occur in the restriction
$V|_S$ since it is the sum of four Galois conjugates of the 3-dimensional
orthogonal group, and thus the elements of order~41 have a 4-dimensional
fixed space. Furthermore, by \cite[Cor.~4.5]{AJL} there are no non-trivial
extensions between the trivial module and the tensor product of the natural
$S$-module with its double Frobenius twist. Since $V$ is self-dual, this
implies that $S$ has a 1-dimensional fixed space on $V$, so in fact $S$ is
contained in the stabilizer of this one-space. But the stabilizer of the
1-space centralized by $x$ is a $\Spin_8(3)$, whence $S$ has to be
contained in a subgroup $\Spin_8$ and cannot be maximal.
\end{proof}

\begin{thm}   \label{thm:triple}
 Let $G$ be a simple exceptional group of Lie type different from $\tw3D_4(q)$
 and $C$ the conjugacy class of a generator of the cyclic subgroup given in
 Table~\ref{tab:exctorus}. Then there exist $x_1,x_2,x_3\in C$ with
 $G=\langle x_1,x_2\rangle$ and $x_1x_2x_3=1$.
\end{thm}

\begin{proof}
The proof is very similar to that of our result in
\cite[Prop.~3.4 and~3.5]{GM10}.
We either compute the structure constant $n(C,C,C)$ in $G$ from the known
generic character tables \cite{Chevie}, or estimate it from below using
\cite[Prop.~3.3]{GM10} to be at least $\half|G|/|T|^2$. \par
We illustrate
this on one of the more difficult cases, viz.~$E_6(q)$. All elements of $T$
whose order does not divide $\Phi_3$ are regular. The other non-central ones
have centralizer $\Phi_3.\tw3D_4(q)/d$. Thus, by \cite[Lemma~3.2]{GM10} the
irreducible characters of $H=E_6(q)_\ad$ not vanishing on $C$ lie in
Lusztig series $\cE(H,s)$ where $s\in H^*=E_6(q)_\SC$ either is regular of
order dividing
$\Phi_3\Phi_{12}$, or $s$ has centralizer $\Phi_3.\tw3D_4(q)$, or $s\in Z(H^*)$.
The characters in the latter series are the extensions to $H$ of the
unipotent characters of $G$. From the degree formulas in \cite[13.9]{Ca} it
follows that just 11 non-trivial unipotent characters do not vanish on $C$,
and their values on $x\in C$ are~$\pm1$. The characters corresponding to $s$
with centralizer $Z=\Phi_3.\tw3D_4(q)$ have degree at least
$|H|_{p'}/|Z|_{p'}\ge q^{20}$ where $p$ is the defining prime of $G$. By
\cite[Prop.~3.3]{GM10} values on $C$ are bounded above by~12, and there are
at most $q^2+q$ such characters. This gives an upper bound for their
contribution to $n(C,C,C)$. Finally, the characters in $\cE(H,s)$ with
$s$ regular semisimple are irreducible Deligne-Lusztig characters, of degree
$|H|_{p'}/|T|_{p'}$, which is roughly~$q^{32}$, and there are less than
$\Phi_3\Phi_{12}/12$ of them. In conclusion, the contribution from the
non-linear characters to the structure constant is (much) less than
$1/2|G|/|T|^2$.
\par
On the other hand, the maximal overgroups of $T$ are known by
Theorem~\ref{thm:excelt}. Clearly, any subgroup $H$ contains at most
$|H|$ triples from $C$ with fixed first component. So it suffices to check
that the sum of orders of relevant maximal subgroups is less than
$\half|G|/|T|^2$. For example in $E_7(q)$, the $E_6$-parabolic subgroups
have order roughly~$q^{107}$, while $|G|/|T|^2$ is roughly~$q^{119}$. In all
other cases, the maximal subgroups are even smaller.
\end{proof}

\begin{prop}   \label{prop:exc1}
 The simple exceptional groups of Lie type admit an unmixed Beauville
 structure.
\end{prop}

\begin{proof}
This is now immediate for all types but $\tw3D_4(q)$. Indeed, for all
exceptional simple
groups $G$ we proved in \cite[Thm.~1.1]{GM10} the existence of a conjugacy
class $C$ such that $G$ is generated by $x_1,y_1\in C$ with
$(x_1y_1)^{-1}\in C$. Theorem~\ref{thm:triple} shows that there exists a
second such class $C'$, and it is immediate to verify that the elements in
$C'$ have order coprime to that of elements in $C$. \par
Thus, to prove the claim it will suffice to exhibit for $G=\tw3D_4(q)$ a second
generating system $(x_2,y_2)$ such that the orders of $x_2,y_2,x_2y_2$ are
prime to the common order $m(G)=q^4-q^2+1$ of the elements in $C$. \par
For $\tw3D_4(2)$ the $(7d,7d,9a)$-structure constant is non-zero, and by
\cite{Atl} no maximal subgroup contains elements from both classes. So now
assume $q>2$. We let $C_1$ be a conjugacy class of regular semisimple
elements of order $\Phi_3$ inside a maximal torus of order $\Phi_3^2$ (such
elements exist for all $q$, see \cite{Chevie}) and $C_2$ a class of regular
semisimple elements of order $\Phi_6\Phi_2\Phi_1$ (such elements exist when
$q\ne2$).
\par
By \cite{Kl3D4} the only maximal subgroups containing an element of order
$\Phi_3$ with centralizer of order dividing $\Phi_3^2$ are $G_2(q)$,
$\PGL_3(q)$, $(\Phi_3\circ\SL_3(q)).2d$ (where $d=\gcd(3,q^2+q+1)$) and the
torus normalizer $\Phi_3^2.\SL_2(3)$. Now note that since $q>2$ none of these
contains elements of order $\Phi_6\Phi_2\Phi_1$ (by Zsigmondy's theorem for
the latter three groups, and since elements of order $\Phi_6$ in the first
subgroup are self-centralizing).
Thus any pair of elements $(x_1,x_2)\in C_1\times C_2$ generates. Using the
result of Gow \cite[Thm.~2]{Gow} or the generic character table one sees that
there exist such pairs with product in any non-trivial semisimple conjugacy
class.
\end{proof}

Alternatively, it would have been possible to choose two classes of regular
semisimple elements whose order is divisible by large Zsigmondy primes, and
a further class containing the product of a long root element with a semisimple
element of order divisible by a third Zsigmondy prime, in such a way that
only the trivial character has non-zero value on all three classes.  Then
the character formula shows that the structure constant
$n(C_1,C_2,C_3)$ equals $|C_2||C_3|/|G|$ and in particular does not
vanish.  \par
From the enumeration of subgroups containing long root elements by
Cooperstein it would then be easy to see that no triple in
$C_1\times C_2\times C_3$ lies in a proper subgroup of $G$. This approach has
been used in \cite[Thm.~8.6]{GM10} to show that $E_8(q)$ satisfies
Theorem~\ref{thm:beau}.

%%%%%%%%%%%%%%%%%%%%%%%%%%%%%%%%%%%%%%%%%%%%%%%%%%%%%%%%%%%%%%%%%%%%%%%%%
\section{Classical groups}   \label{sec:class}

Here we prove Theorem~\ref{thm:beau} for the simple classical groups of
Lie type.

The groups $\PSL_2(q)$, $q\ge7$ were shown to admit a Beauville structure
in \cite{BCG05}. Before treating the generic case it will be convenient to
consider some linear, unitary and symplectic groups of small rank.

\begin{prop}   \label{prop:dim3}
 Theorem~\ref{thm:beau} holds for the simple groups $\PSL_3(q)$,
 $\PSU_3(q)$ and $\PSp_4(q)$, where $q\ge3$.
\end{prop}

\begin{proof}
Let first $G=\PSL_3(q)$. Since $\PSL_3(2)\cong\PSL_2(7)$ we may assume that
$q\ge3$. In \cite[Prop.~3.13]{GM10} we showed that $G$ can be generated by
a pair of elements of order $\Phi_3^*(q)$, with product of the same order.
Let $C_1$ be the conjugacy class of a regular element $x$ of order~$(q^2-1)/d$,
with $d=\gcd(3,q-1)$. Then the only maximal subgroups containing $x$ are
maximal parabolic, and three subgroups $\fA_6$ when $q=4$. Let $C_2$
consist of regular unipotent elements. Then the $(C_1,C_2,C_2)$-structure
constant is non-zero in $G$ by \cite{Chevie}. Let $P$ be a maximal parabolic
subgroup. Then
its derived subgroup contains all unipotent elements, but no element of order
$(q^2-1)/d$. Thus, $P$ cannot contain triples from $C_1\times C_2\times C_2$
with product~1. For $\PSL_3(4)$ it is easy to check that the structure
constant in $G$ is larger than those in the $\fA_6$-subgroups. \par
Similarly, we showed in \cite[Prop.~3.11]{GM10} that $G=\PSU_3(q)$ with
$q\ge3$ can be generated by elements of order $\Phi_6^*(q)$. Again, let
$C_1$ be a class of elements of order $(q^2-1)/d$, where $d=\gcd(3,q+1)$.
The only maximal subgroups containing such elements are the Borel subgroups
and the image in $G$ of a subgroup $\GU_2(q)$ of $\SU_3(q)$. Moreover, for
$q=3,4,5$ there is an additional class of subgroups $4^2.\fS_3$, $5^2.\fS_3$
respectively three classes of $M_{10}$. \par
Choosing $C_2$ to consist of regular unipotent elements, one checks that the
$(C_1,C_2,C_2)$-structure constant is non-zero \cite{Chevie}. By quotienting
out the normal closure of a Sylow $p$-subgroup we see that such triples cannot
lie in either of the two generic maximal subgroups. For $q=3,4,5$ direct
computation in $G$ shows that there exist generating triples.  \par
For $G=\PSp_4(q)$, $q\ge3$ we produced in \cite[Prop.~3.8]{GM10} a generating
triple with elements of order $\Phi_4^*(q)$. A direct computation shows that
$\PSp_4(3)$ contains generating triples with orders $(9,9,8)$, and these
are prime to $(3^2+1)/2=5$.
For $q\ge4$ let $C_1$ denote a class of regular elements of order~$(q+1)/d$
inside a maximal torus of order $(q+1)^2/d$, where $d=\gcd(2,q-1)$, and $C_2$
a class of regular unipotent elements. It follows from the known character
table \cite{Sr68} that $n(C_1,C_2,C_2)>\frac{1}{4}q^6$. On the other hand, by
\cite[Thm.~5.6]{KaLi82} for example, 
the only maximal subgroups of $G$ containing elements from class $C_1$ are
the normalizers of subgroups $\SO_4^+(q)$ and $\OO_3(q)\times\OO_2^-(q)$.
In the latter, any unipotent element is centralized by a 1-dimensional
torus, so it does not contain regular unipotent elements. The structure
constant in $H=\SO_4^+(q)$ can be estimated from above by the number of regular
unipotent elements of $G$ contained in $H$, which is less than $2q^4$. So
there exist generating triples.
\end{proof}

\begin{table}[htbp] 
  \caption{Maximal tori in some linear and unitary groups}
  \label{tab:unitor}
\[\begin{array}{|ll||lll|} 
\hline
 G& \text{\cite{GM10}}& T_1& T_2& T_3\\
\skipa \hline \hline
 \SL_4(q)& \Phi_3^*&   (q^2+1)(q+1)& (q^2-1)(q+1)& (q^2-1)(q-1)\\
 \SU_4(q)& \Phi_6^*&   (q^2+1)(q-1)& (q^2-1)(q-1)& (q^2-1)(q+1)\\
 \SU_5(q)& \Phi_{10}^*& q^4-1& (q^3+1)(q-1)& (q^3+1)(q+1)\\
 \SU_6(q)& \Phi_{10}^*& \Phi_1\Phi_3\Phi_6& \Phi_1^2\Phi_2\Phi_4& \Phi_2\Phi_6^2\\
\hline
\end{array}\]
\end{table}

\begin{prop}   \label{prop:U456}
 Theorem~\ref{thm:beau} holds for the simple groups $\PSL_4(q)$, $\PSU_4(q)$,
 $\PSU_5(q)$ and $\PSU_6(q)$.
\end{prop}

\begin{proof}
For $G=\SL_4(q)$ we may assume that $q>3$ since $\SL_4(2)\cong\fA_8$ is an
alternating group and, by explicit computation, $\PSL_4(3)$ has generating
triples of elements of order~5.
Let $C_i$ be the class of a regular semisimple element
$x_i$ in a maximal torus $T_i$ as in Table~\ref{tab:unitor}, $i=1,2,3$. Such
classes exist whenever $q\ge 4$. Then $x_1$ acts irreducibly, and there exists
a Zsigmondy prime for $o(x_1)$, so the maximal subgroups containing a triple
$(x_1,x_2,x_3)\in C_1\times C_2\times C_3$ are described in
\cite[Thm.~2.2]{GM10}. The only ones containing Singer cycles are
the normalizers of $\GL_2(q^2)\cap\SL_4(q)$ and $\GU_2(q^2)\cap\SL_4(q)$.
But in the first group, any semisimple element has centralizer order divisible
by either $(q^4-1)/(q-1)$ or by $(q+1)(q^2-1)$, while $x_3$ has only
centralizer order $(q-1)(q^2-1)$. Similarly we can exclude the second case
using the order of $x_2$. Thus any triple $(x_1,x_2,x_3)$ generates $G$.
By \cite[Thm.~2]{Gow} there exist such triples with product~1. \par
The element orders are coprime to those of the element order chosen in
\cite[Prop.~3.13]{GM10}, so the Beauville property is satisfied, and we
conclude by passing to the central quotient $G/Z(G)=\PSL_4(q)$. 
%%But elements of order $q-1$ in the torus of order $q^3-1$ of
%%$\SL_4(q)$ have centralizer $\GL_3(q)$ and hence three equal eigenvalues,
%%while the classes $C_2$ and $C_3$ can be chosen such that their powers of
%%order $q-1$ have two eigenvalues of multiplicity~2 each.
\par
For $G=\SU_4(q)$ regular elements as indicated in Table~\ref{tab:unitor} exist
whenever $q\ge4$. The argument is then completely analogous; again the only
subgroups possibly containing elements from $C_1$ is the normalizer of
$\GL_2(q^2)\cap\SU_4(q)$ and the stabilizer of an isotropic point (a maximal
parabolic subgroup with Levi factor $\GL_2(q^2)\cap\SU_4(q)$), which can be
excluded as before. The group
$\SU_4(2)\cong\PSp_4(3)$ was treated in Proposition~\ref{prop:dim3}, the
group $\SU_4(3)$ has generating triples of elements of order~20.
\par
For $G=\SU_5(q)$ choose $C_i$ to contain regular elements from the
tori $T_i$ in Table~\ref{tab:unitor}, which exist whenever $q\ge3$. Here
the only maximal subgroups of order divisible by $o(x_i)$, $i=1,2,3$, are
$\GU_4(q)$. But there the centralizer order of a semisimple element of order
$q^3+1$ is divisible by $(q^3+1)(q+1)$, while the element $x_3$ has smaller
centralizer order $(q^3+1)(q-1)$. We now conclude
as before. The group $\SU_5(2)$ has generating triples of order~15.
\par
For $G=\SU_6(q)$ we again choose $C_i$ to contain regular elements from the
tori $T_i$ in Table~\ref{tab:unitor}, which exist whenever $q\ge 3$. Here
the only maximal subgroups of order divisible by $o(x_i)$, $i=1,2,3$, are
the normalizers of $\GL_3(q^2)\cap\SU_6(q)$. But there the centralizer order
of a semisimple element of order $q^3+1$ is divisible by $(q^6-1)/(q+1)$,
while the element $x_3$ has centralizer order $(q^3+1)^2/(q+1)$. We now
conclude as before. The group $\SU_6(2)$ has generating triples of order~10.
Again, the element orders are coprime to those in \cite{GM10}, so the
Beauville property is satisfied.
%%Now elements of order~$q+1$ in the torus of order $q^5+1$ of $\SU_6(q)$ have
%%five equal eigenvalues, while we can arrange so that elements of order~$q+1$ in
%%the cyclic subgroup generated by $x\in C_i$, $i=2,3$, do not have this
%%property, .
\end{proof}

For the remaining classical groups, we choose conjugacy classes $C_1$, $C_2$
of regular semisimple elements of orders as given in
Table~\ref{tab:classtorus}. Here $k^\eps$ is shorthand for $q^k-\eps1$, and
$k^\eps\opl (n-k)^\delta$ denotes an element which acts as $k^\eps$ on a
subspace of dimension~$k$ and as $(n-k)^\delta$ on a complementary subspace
of dimension~$n-k$ in the linear and unitary case, respectively twice the
dimensions in the symplectic and orthogonal cases. We have also indicated
the orders of elements
in the generating triple constructed in \cite{GM10}. It is straightforward to
check that the cyclic subgroups chosen here intersect those from
\cite{GM10} trivially.

\begin{table}[htbp] 
  \caption{Elements in classical groups}
  \label{tab:classtorus}
\[\begin{array}{|ll||l|ll|} 
\hline
 G&& \text{\cite{GM10}}& \ \ x_1& \ \ x_2\\
\skipa \hline \hline
 \SL_n(q)& n\ge5\text{ odd}&     \Phi_n^*& (n-1)^+& (n-2)^+\opl 2^+\\
         & n\ge6\text{ even}&\Phi_{n-1}^*& (n-2)^+& (n-3)^+\opl 3^+\\ \hline
 \SU_n(q)& n\ge7\text{ odd}&  \Phi_{2n}^*& (n-1)^+& (n-4)^-\opl 4^+\\
  & n\ge8\text{ even}&      \Phi_{2n-2}^*& (n-3)^-\opl3^-& (n-5)^-\opl 5^-\\ \hline
      \Spin_7(q)&  &         \Phi_6^*&     3^+& 2^-\opl1^+\\
 \Spin_{2n+1}(q)& n\ge4&  \Phi_{2n}^*&  (n-2)^-\opl2^-& (n-3)^-\opl3^+\\ \hline
   \Sp_{2n}(q)& n\ge3&    \Phi_{2n}^*&     n^+& (n-1)^-\opl1^-\\ \hline
  \Spin_8^+(q)&        &          3^-& 4^+& 4^+\\
 \Spin_{10}^+(q)&      &          4^-& 3^-\opl 2^-& 5^+\\
 \Spin_{2n}^+(q)& n\ge6&      (n-1)^-& (n-2)^-\opl 2^-& (n-3)^-\opl 3^-\\ \hline
 \Spin_{2n}^-(q),& n\ge4&  \Phi_{2n}^*& (n-1)^-& (n-3)^-\opl 3^+\\
\hline
\end{array}\]
\end{table}

In order to verify generation in classical groups, the following result will
be useful and may be of independent interest. It follows by the main result
of \cite{GPPS} and inspection of the tables in that paper. 

\begin{thm}   \label{twozsig}
 Let $G=\GL(V)=\GL_n(q)$ where $q=p^a$ with $p$ prime. Assume that $n > 4$.
 Suppose that $H$ is an irreducible subgroup of $G$ containing elements of
 orders $r_i$, $i=1,2$ where $r_i$ is a Zsigmondy prime divisor of $q^{e_i}-1$
 where $e_1 > e _2 > n/2$. Then one of the following holds:
 \begin{enumerate}
  \item  $H$ contains $\SL(V)$, $\SU(V)$, $\Omega^{(\pm)}(V)$ or $\Sp(V)$; 
  \item  $H$ preserves an extension field structure on $V$ (of degree $f$
   dividing $\gcd(n,e_1,e_2)$);
  \item  $H$ normalizes $\GL_n(p^b)$ for some $b$ properly dividing $a$;
  \item  $H\le\GL_1(q)\wr\fS_n$ is imprimitive;
  \item  $H$ normalizes $\fA_{n+1+\delta}$ where $\delta=1$ if $\gcd(p,n)\ne 1$
  and $0$ otherwise; 
  \item  $n=5$,  $H=M_{11}$, $(e_1,e_2)=(5,4)$ and $q=3$;
  \item  $n=6$,  $H=2.M_{12}$, $(e_1,e_2)=(5,4)$ and $q=3$; 
  \item  $n=6$ or $7$,  $(e_1, e_2)=(6,4)$ and $q$ is prime; or 
  \item  $n=11$, $H=M_{23}$ or $M_{24}$, $(e_1,e_2)=(11,10)$ and $q=2$; 
 \end{enumerate}
\end{thm}

We can get rid of some of these examples with a stronger hypothesis. This
follows by the previous result and the computation of $\Phi_{e_i}^*(q)$,
using \cite[Lemma~2.1]{GM10}, for example:

\begin{cor} \label{cor:twozsig2}
 Let $G=\GL(V)=\GL_n(q)$ where $q=p^a$ with $p$ prime. Assume that $n > 4$.
 Suppose that $H$ is an irreducible subgroup of $G$ containing elements of
 orders $\Phi_{e_i}^*(q)>1$, $i=1,2$ where $e_1>e_2>n/2$. Then one of the
 following holds:
   \begin{enumerate}
  \item  $H$ contains $\SL(V)$, $\SU(V)$, $\Omega^{(\pm)}(V)$ or $\Sp(V)$; 
  \item  $H$ preserves an extension field structure on $V$ (of degree $f$
   dividing $\gcd(n,e_1,e_2)$);
  \item  $H$ normalizes $\GL_n(p^b)$ for some $b$ properly dividing $a$;
  \item  $n=6$, $H=2.\PSL_3(4)$, $(e_1,e_2)=(6,4)$ and $q=3$; 
  \item  $n=7$, $H=\Sp_6(2)$, $(e_1,e_2)=(6,4)$ and $q=3$; or
  \item  $H$ normalizes $\fA_{n+1+\delta}$ where $\delta=1$ if $\gcd(p,n)\ne 1$
  and $0$ otherwise, or $H\le\GL_1(q)\wr\fS_n$ is imprimitive, and either
  \begin{enumerate}
   \item $q=2$, $(e_1,e_2)\in\{(12,10),(18,10),(18,12)\}$ and
    $e_1\le n<2e_2$; or
   \item $q=3$, $(e_1,e_2)=(6,4)$ and $n=6,7$.
  \end{enumerate}
 \end{enumerate}
\end{cor}

\begin{prop}   \label{prop:L}
 Theorem~\ref{thm:beau} holds for the linear groups $\PSL_{n}(q)$, $n\ge5$.
\end{prop}

\begin{proof}
We argue in $G=\SL_n(q)$. We choose conjugacy classes $C_1$, $C_2$ in $G$
of regular semisimple elements $x_i$ of orders as given in
Table~\ref{tab:classtorus}, and we let $C_3$ be any class of semisimple elements
of order prime to $\Phi_n^*(q)$ when $n$ is odd, respectively to
$\Phi_{n-1}^*(q)$
when $n$ is even. First assume that Zsigmondy primes $r_i$ exist for the
factors $\Phi_{e_i}(q)$ of $o(x_i)$ with $e_i>n/2$, $i=1,2$. Then under each
$x_i\in C_i$, $i=1,2$, the natural module of $G$ splits into two irreducible
submodules of incompatible dimensions, so the subgroup
$H:=\langle x_1,x_2\rangle$ generated by any pair of elements $x_i\in C_i$ is
irreducible. We claim that $H=G$. Otherwise, when $n\ne6$ we are in one of
the cases of Corollary~\ref{cor:twozsig2}. Now note that $H$ cannot be an
extension field subgroup since $\gcd(e_1,e_2)=1$, and it can't be a subfield
subgroup by looking at a suitable Zsigmondy prime divisor of $o(x_i)$. Also,
no proper classical subgroup has order divisible by both $o(x_1)$ and $o(x_2)$.
Since $e_2=e_1-1$, we're not in cases (4)--(6). \par
When $n=6$ then $e_2=n/2$, so Corollary~\ref{cor:twozsig2} is not applicable.
Still, by \cite[Lemma~2.1 and Thm.~2.2]{GM10} we get the same conclusion as
before unless $(n,q)=(6,2)$. In the latter case replace the second class by
an irreducible element of order $63$ and let $C_3$ consist of elements of
order $7$ with a $3$-dimensional fixed space. By inspection of the possible
overgroups in \cite{GPPS} the $x_i$ generate.
In the Zsigmondy exception $(n,q)=(7,2)$, replace $C_1$ by the class of a
regular element of type $4^+\opl3^+$; when $(n,q)=(8,2)$, replace it by the
class of a Singer cycle, of order $(q^8-1)/(q-1)$. Then the previous arguments
apply. \par
By \cite[Thm.~2]{Gow} there exists $x_3\in C_3$ with $x_1x_2x_3=1$.
Now consider the image of this triple in the simple group $\PSL_n(q)$. The
element orders are coprime to that for the generating triple exhibited in
\cite[Prop.~3.13]{GM10} which proves the existence of a Beauville structure.
\end{proof}

\begin{prop} \label{prop:U}
 Theorem~\ref{thm:beau} holds for the unitary groups $\PSU_{n}(q)$, $n\ge 7$.
\end{prop}

\begin{proof}
As before, we work in $G=\SU_n(q)$ and let
$C_1,C_2$ contain regular semisimple elements $x_1,x_2$ of the orders
indicated in Table~\ref{tab:classtorus}. For $n=8$ replace $C_2$ by a class
of regular semisimple elements of order $(q^8-1)/(q+1)$. For $n=7$ let $C_3$
be a class of regular semisimple elements of type $5^-\opl2^+$, and otherwise
let it be any class of semisimple elements of order prime to
$\Phi_{2n}^*(q)$ when $n$ is odd, respectively to $\Phi_{2n-2}^*(q)$ when $n$ is
even. Then $H=\langle x_1,x_2,x_3\rangle$, with $x_i\in C_i$, acts irreducibly
on the natural module for $G$. For each $i$ let $e_i$ be maximal with the
property that
$o(x_i)$ has a corresponding Zsigmondy prime. Then $e_i>n$ for at least one
$i$. Thus, if $H$ is proper then by \cite[Thm.~2.2]{GM10} either the Zsigmondy
primes for both factors are small or the possible overgroups are classical,
extension or subfield groups. The latter three classes can be excluded by
using the fact that we have two distinct Zsigmondy primes. The first situation
only arises when $(n,q)\in\{(7,2),(8,2)\}$ by \cite[Lemma~2.1]{GM10}, but the
conclusion still holds by \cite{GPPS}.

So $H=G$ in all cases. Now the existence of triples
$(x_1,x_2,x_3)\in C_1\times C_2\times C_3$ with product~1 follows from
\cite{Gow}. Passing to the quotient $\PSU_n(q)$ we obtain the desired result,
noting again that the element orders are coprime to that in
\cite[Prop.~3.12]{GM10}.
\end{proof}

\begin{prop}   \label{prop:Sp}
 Theorem~\ref{thm:beau} holds for the symplectic groups $\PSp_{2n}(q)$,
 $n\ge3$, $q$ odd.
\end{prop}

\begin{proof}
Let $G=\Sp_{2n}(q)$, $C_1$ and $C_2$ conjugacy classes of regular semisimple
elements as indicated in Table~\ref{tab:classtorus} and $C_3$ any class
of semisimple elements of order prime to $\Phi_{2n}^*(q)$. Then for any pair
$(x_1,x_2)\in C_1\times C_2$ of elements, the subgroup
$H:=\langle x_1,x_2\rangle$ acts irreducibly on the natural module. If
$H$ is proper, then by \cite[Lemma~2.1 and Thm.~2.2]{GM10}, either
$(n,q)=(4,3)$ or $H$ is contained in an extension or subfield subgroup. The
latter cases do not occur by consideration of suitable Zsigmondy prime
divisors. When $(n,q)=(4,3)$, we choose $C_3$ to be a class of elements of
order divisible by $q^3-1$; its Zsigmondy prime~13 gives no exception to
\cite[Thm.~2.2]{GM10}. So we have $H=G$ in all cases. By \cite{Gow} there
exist triples from the chosen conjugacy classes with product~1. \par
In \cite[Prop.~3.8]{GM10} we produced a
generating triple for $\PSp_{2n}(q)$ consisting of elements of order
$\Phi_{2n}^*(q)$. This is coprime to the orders of $x_1,x_2$, so the
proof is complete. 
\end{proof}

\begin{prop}   \label{prop:Oodd}
 Theorem~\ref{thm:beau} holds for the orthogonal groups $\OO_{2n+1}(q)$,
 $n\ge3$.
\end{prop}

\begin{proof}
Let $G=\Spin_{2n+1}(q)$, $C_1,C_2$ conjugacy classes in $G$ of regular 
semisimple elements of orders as given in Table~\ref{tab:classtorus} and $C_3$
any class of semisimple elements of order prime to $\Phi_{2n}^*(q)$,
respectively a class of elements of order~$q^{n-1}+1$ when $4\le n\le6$.
For any pair $(x_1,x_2)\in C_1\times C_2$ of elements, the subgroup
$H:=\langle x_1,x_2\rangle$ either acts irreducibly on the natural module,
or it has a composition
factor of dimension $2n$. In the latter case, $H$ is contained in the
stabilizer of an anisotropic line, so in a $2n$-dimensional orthogonal group.
But the first element does not lie in an orthogonal group of minus type, the
second not in one of plus type, unless $n=3$ and $q\le4$. We return to these
cases later. So otherwise $H$ is irreducible.
Now note that for $(n,q)\ne(4,2)$, one of the two (respectively three when
$n=4,5,6$) element orders is divisible by a Zsigmondy prime divisor of $q^e-1$
with $2e>2n+1$. Moreover, for $n\ge5$ there are even two different such $e$.
If $H$ is proper, then by \cite[Lemma~2.1 and Thm.~2.2]{GM10} respectively
Corollary~\ref{cor:twozsig2} either $H$ is contained in an extension or
subfield subgroup or we have
$(n,q)\in\{(3,2),(4,2),(4,3),(8,2)\}$. Since none of the groups
in Corollary~\ref{cor:twozsig2}(6)(a) contains elements of order $2^{12}-1$,
$(n,q)=(8,2)$ is no exception. The other three groups will be considered
later. The extension and subfield subgroups can be excluded by using suitable
Zsigmondy primes.

In \cite[Prop.~3.8]{GM10} we produced a generating triple for $G=\OO_{2n+1}(q)$
consisting of elements of order $\Phi_{2n}^*(q)$, which is prime to the
orders chosen here.

Finally, consider $\OO_7(q)$ with $q \le 4$ and $\OO_9(q)$ with $q\le3$.
Explicit computation shows that
$\OO_7(2)=\PSp_6(2)$ contains generating triples of order~7, $\OO_7(3)$
contains generating triples of order~13, and $\OO_7(4)=\PSp_6(4)$ contains
generating triples of order~17. The group $\OO_9(2)=\PSp_8(2)$ contains
generating triples of order~7, $\OO_9(3)$ contains generating triples of
order~13. 
\end{proof}

\begin{prop}   \label{prop:O-}
 Theorems~\ref{thm:beau} holds for the orthogonal groups $\OO_{2n}^-(q)$,
 $n\ge4$.
\end{prop}

\begin{proof}
Let $G=\Omega_{2n}^-(q)$. Let $C_1,C_2$ consist of regular semisimple elements
of types as in Table~\ref{tab:classtorus} and $C_3$ any class of semisimple
elements of order prime to $\Phi_{2n}^*(q)$. Then any pair of elements
$x_i\in C_i$, $i=1,2$, necessarily generates an irreducible subgroup
$H:=\langle x_1,x_2\rangle$ of $G$, unless possibly when $n=4$. But in the
latter case it is easy to see that no reducible subgroup has order divisible
by Zsigmondy primes for $q^3+1$ and for $q^3-1$, which both exist when $q\ne2$.
We exclude $(n,q)=(4,2)$ for the moment. Otherwise, at least one of the two
element orders $o(x_i)$ is divisible by a Zsigmondy prime divisor of $q^e-1$
with $e>n$. Moreover, for $n\ge7$ there are even two different such $e$. Thus
by \cite[Lemma~2.1 and Thm.~2.2]{GM10} and Corollary~\ref{cor:twozsig2} we
have $H=G$ unless $(n,q)\in\{(5,2),(6,2),(4,4)\}$. In the
latter cases, let $C_3$ contain elements of order divisible by $2^5-1$ when
$n=5,6$, by $4^2+1$ when $(n,q)=(4,4)$, then we still have generation for
any triple $(x_1,x_2,x_3)\in C_1\times C_2\times C_3$. By \cite{Gow}, we can
find $x_i \in C_i$ with product $1$.\par
The previously excluded group $\OO_8^-(2)$ is generated by a triple from
$(21a,21a,30a)$.
%% the group $\OO_{10}^-(2)$ by a triple from $(17a,17a,35a)$,
%% so we need not worry about this.
Combining this with \cite[Prop.~3.6]{GM10} the claim follows as in the previous
cases.
\end{proof}

\begin{prop}   \label{prop:O+}
 Theorem~\ref{thm:beau} holds for the orthogonal groups $\OO_{2n}^+(q)$,
 $n\ge4$.
\end{prop}

\begin{proof}
We argue in $G=\Omega_{2n}^+(q)$. 
First assume that $n > 6$. Let $C_1$ consist of elements with precisely
two invariant subspaces of dimensions $4$ and $2n-4$ and $C_2$
consist of elements with precisely two invariant subspaces of dimensions
$6$ and $2n-6$. Moreover, assume that the orders of the elements in $C_i$
are divisible by all Zsigmondy prime divisors of $q^{2n-4}-1$ and $q^{2n-6}-1$,
respectively. Let $C_3$ be any class of semisimple elements of order prime to
$q^{n-1}+1$. By Corollary~\ref{cor:twozsig2}, there are no maximal
subgroups containing elements from both $C_1$ and $C_2$. By \cite{Gow}, we
can choose $x_i \in C_i$ with product $1$. Argue as usual to complete the proof.
 
If $n=6$, let $C_1$ and $C_2$ be as above. Now we apply \cite[Thm.~2.2]{GM10}
instead and argue the same way as long as there is a Zsigmondy prime divisor
$r$ of $q^8-1$ with $r>17$. This only fails for $q=2$. It can be checked with
GAP that $\OO_{12}^+(2)$ has a generating triple consisting of elements of
order~17. Argue as above to complete the proof.
 
If $n=5$, let $C_1$ be as above. Let $C_2$ consist of elements of order
$(q^5-1)/(2,q-1)$ and $C_3$ any class of semisimple elements of order prime
to $q^4+1$. Apply \cite[Thm.~2.2]{GM10} to conclude that there are
no maximal subgroups intersecting both $C_1$ and $C_2$ unless possibly
$q \le 5$. If $q=5$, inspection of the maximal subgroups shows the result is
still true. By explicit computation the group
$\OO_{10}^+(2)$ contains generating triples of elements of order~31, the
group $\OO_{10}^+(3)$ has generating triples of order~121, the group
$\OO_{10}^+(4)$ has generating triples of order~341.
Argue as above to complete the proof.

If $n=4$ and $q > 2$,  let $C_1$ be a conjugacy class of regular semisimple
elements of order $(q^4-1)/(2,q-1)$. Let $C_2$ and $C_3$ be the twists of
$C_1$ by the triality automorphism and its square. By Kleidman \cite{Kl87},
one sees that no
maximal subgroup of $\OO_8^+(q)$ intersects each of the $C_i$. By \cite{Gow},
there exist $x_i \in C_i$ with product $1$ and they generate by the previous
remarks.
 
By explicit computation the group $\OO_8^+(2)$ contains generating triples of
elements of order~7.
 
In \cite{GM10}, we showed that there are generating triples of elements of
$\OO_{2n}^+(q)$ in a class $C$ of regular semisimple elements of order dividing
$\Phi_{2n-2}^*(q)(q+1)$ lying in a maximal torus of order
$(q^{n-1}+1)(q+1)/\gcd(4,q^{n-1}+1)$. Comparing the fixed spaces of elements
in that torus with those in classes $C_1,C_2$ we conclude that the required
intersection property holds.
\end{proof}

%%%%%%%%%%%%%%%%%%%%%%%%%%%%%%%%%%%%%%%%%%%%%%%%%%%%%%%%%%%%%%%%%%%%%%%%%
\section{Alternating and sporadic groups}   \label{sec:altspor}

The existence of unmixed Beauville structures for all alternating groups
$\fA_n$ with $n\ge6$ was proved by Fuertes--Gonz\'alez-Diez \cite[Thm.~1]{FG}
after asymptotic results had been obtained by
Bauer--Catanese--Grunewald \cite{BCG05}. So the proof of Theorem~\ref{thm:beau}
is complete once we've shown the following:

\begin{prop}   \label{prop:spor}
 The sporadic simple groups and the Tits group admit an unmixed Beauville
 structure.
\end{prop}

\begin{table}[htbp] 
  \caption{Conjugacy classes for sporadic groups} \label{tab:spor}
\[\begin{array}{|l|rr||l|rr||l|rr|} 
\hline
 G& C_1& C_2&   G& C_1& C_2&  G& C_1& C_2\\
\hline\hline
 M_{11}&      5a&   8a&   M_{24}& 11a& 21a&       HN&  22a& 35a\\
 M_{12}&      3b&  10a&      McL&  9a& 14a&       Ly&  31a& 37a\\
    J_1&      7a&  11a&       He& 14c& 15a&       Th&  13a& 31a\\
 M_{22}&      7a&   8a&       Ru& 16a& 26a&  Fi_{23}&  13a& 23a\\
    J_2&      8a&  15a&      Suz& 11a& 21a&     Co_1&  23a& 33a\\
 M_{23}&     11a&  14a&       ON& 19a& 16a&      J_4&  31a& 37a\\
\tw2F_4(2)'& 10a&  16a&     Co_3& 21a& 22a&  Fi_{24}'& 23a& 33a\\
     HS&      7a&  20a&     Co_2& 11a& 28a&        B&  23a& 31a\\
    J_3&     12a&  17a&  Fi_{22}& 21a& 22a&        M&  47a& 59a\\
\hline
\end{array}\]
\end{table}

\begin{proof}
In Table~\ref{tab:spor} we give for each sporadic group $G$ two conjugacy
classes $C_1,C_2$ such that the structure constant $n(C_1,C_1,C_2)$ is
non-zero and moreover no maximal subgroup of $G$ has non-trivial intersection
with both classes. This is easily checked from the known lists of
maximal subgroups, see \cite{Atl} respectively the Atlas homepage.
For the group $J_2$ we used explicit computation in the 6-dimensional
representation over $\FF_4$ and for $\tw2F_4(2)'$ in the permutation
representation on 1600 points to check for generating triples. Since the
element orders in the triples in Table~\ref{tab:spor} are prime to those in
\cite[Prop.~4.5]{GM10}, there do exist corresponding unmixed Beauville
structures.
\end{proof}

%%%%%%%%%%%%%%%%%%%%%%%%%%%%%%%%%%%%%%%%%%%%%%%%%%%%%%%%%%%%%%%%%%%%%%%%%
\section{Bounds on character values}   \label{sec:bound}

We will prove the following result which may be of independent interest.

Let $W$ be an irreducible Weyl group and $\rho$ a graph automorphism of
its Dynkin diagram. Then there exists a constant $C=C(W,\rho)$ with the
following property: whenever $\bG$ is a connected reductive algebraic group
with Weyl group $W=W(\bG)$ and $F:\bG\rightarrow\bG$ a Steinberg map inducing
the graph automorphism $\rho$ on $W$, with group of fixed points $G:=\bG^F$,
then for every regular semisimple element $s\in G$ and any irreducible
character $\chi\in\Irr(G)$ we have $|\chi(s)|<C$.  \par
In fact we'll show a version which allows for $\bG$ to be disconnected.
Since we'll need some ingredients on characters of disconnected groups which
are not yet available in the literature, we start by setting up some notation
first. \par
Let $\bG$ be an algebraic group with connected component of the identity
$\bG^\circ$. We assume throughout that $\bG/\bG^\circ$ is cyclic and that
all elements of $\bG/\bG^\circ$ are semisimple. Let
$F:\bG\rightarrow\bG$ be a Steinberg map on $\bG$ with trivial action on
$\bG/\bG^\circ$ and $G:=\bG^F$, $G^\circ:=(\bG^\circ)^F$. For $\bT^\circ$
an $F$-stable maximal torus of $\bG^\circ$ contained in a not
necessarily $F$-stable Borel subgroup $\bB^\circ$ of $\bG^\circ$ we set
$\bB:=N_\bG(\bB^\circ)$ and $\bT:=N_\bB(\bT^\circ)$. Following \cite{DM94}
we say that $\bT$ is a \emph{maximal ``torus'' of $\bG$} (note that this is
not in general a torus!). Since all Borel subgroups of $\bG^\circ$ and all
maximal tori of $\bB^\circ$ are conjugate in $\bG^\circ$ respectively
$\bB^\circ$ we have $\bT/\bT^\circ\cong \bB/\bB^\circ\cong \bG/\bG^\circ$. 
Let $g\in\bT$ generate $\bG/\bG^\circ$. Since by assumption
$g^{-1}F(g)\in\bT^\circ$ and $\bT^\circ$ is connected, there exists by the
theorem of Lang--Steinberg an element $h\in\bT^\circ$ such that
$g^{-1}F(g)=h^{-1}F(h)$, so that $\sigma:=gh^{-1}\in\bT$ is $F$-stable and
generates $\bG/\bG^\circ$, hence $G/G^\circ$.
\par
For $\bU\le\bB$ the unipotent radical of $\bB$, define
$$Y:=Y(\bU^F):=\{x\in\bG\mid x^{-1}F(x)\in\bU\}.$$
This variety has commuting
actions of $G$ from the left and $T:=\bT^F$ from the right by multiplication,
so its $\ell$-adic cohomology groups with compact support $H_c^i(Y)$ are
$G\times T$-bimodules. For $\theta\in\Irr(T)$ we let $H_c^i(Y)_\theta$ denote
the $\theta$-isotypic component for the right $T$-action. Then the generalized
character
$$R_{T,\theta}(g):=\tr\left(g|H_c^*(Y)_\theta\right)\qquad(g\in G)$$
of $G$ constitutes an analogue of Deligne--Lusztig induction for the
disconnected group $\bG$ which has been studied by Digne--Michel \cite{DM94}.
We'll need the following property:

\begin{prop}[Disjointness]   \label{prop:disjoint}
 Let $\bT\le\bG$ be a maximal ``torus'' of $\bG$, $T:=\bT^F$ and
 $\theta_i\in\Irr(T)$, $i=1,2$. If the virtual $\bG^F$-characters
 $R_{T,\theta_1},R_{T,\theta_2}$ have an irreducible constituent in common,
 then there exists $g\in N_\bG(\bT^\circ)^F$ such that
 $\theta_1^g|_{T^\circ}=\theta_2|_{T^\circ}$.
\end{prop}

\begin{proof}
When $\bG$ is connected, this is well-known \cite[Prop.~13.3]{DM91}. We reduce
to that case. By construction of $\sigma$, the set of powers
$S:=\{\sigma^j\mid 0\le j\le[\bG:\bG^\circ]-1\}\subset \bT^F$ forms a system
of coset representatives of $\bG/\bG^\circ$. \par
The Deligne--Lusztig variety $Y$ then decomposes into a disjoint
union of open and closed subsets
$$Y=\coprod_{g\in S} Y_g,\qquad\text{where }
  Y_g:=\{gx\in g\bG^\circ\mid x^{-1}F(x)\in \bU\},$$
so
$$H_c^i(Y)=\bigoplus_{g\in S} H_c^i(Y_g)$$
by \cite[Prop.~10.7(ii)]{DM91}. Note that $Y_g$ is isomorphic to $Y_1$ via
$$Y_1\longrightarrow Y_g,\qquad x\mapsto gx.$$
Here, $Y_1$ is just the ordinary Deligne-Lusztig variety for the torus
$T^\circ$ in the connected group $\bG^\circ$. Moreover, $h\in G^\circ$ acts
on $Y_g$ on the left as $h^g$ does on $Y_1$, and the right $T^\circ$-actions
on $Y_1,Y_g$ commute with the above isomorphism.  \par
If $H_c^i(Y)_{\theta_1}$ and $H_c^j(Y)_{\theta_2}$ have a common
$\bG^F$-constituent, then clearly there is also a common $G^\circ$-constituent
of $H_c^i(Y_g)_{\theta_1}$ and $H_c^j(Y_h)_{\theta_2}$ for some $g,h\in S$.
But then by \cite[Prop.~13.3]{DM91} the pairs
$(\bT^\circ,\theta_1^g),(\bT^\circ,\theta_2^h)$ are geometrically
$\bG$-conjugate, where we identify $\theta_i$ with its restriction to
$T^\circ$. This is the claim.
\end{proof}

For $\bT\le\bG$ a maximal torus, let us set
$T_0^\circ:=C_\bT^\circ(\sigma)^F$ and $G_0:=C_\bG^\circ(\sigma)^F$. Note that
when $\bG=\bG^\circ$ is connected then $T_0^\circ=T$ and $G_0=G$.

\begin{defn}   \label{def:regular}
 A semisimple element $s\in\bG$ is called \emph{regular} if it lies in a
 unique maximal ``torus'' of $\bG$, which happens if and only if its connected
 centralizer is a (true) torus of $\bG^\circ$ (which, in general, will not be
 a maximal torus of $\bG^\circ$.)
\end{defn}

\begin{lem}   \label{lem:charfun}
 Let $s\in G^\circ\sigma$ be regular semisimple in the maximal ``torus''
 $T$, so $C:=C_G(s)=C_T(s)=C_T(\sigma)$. Then the characteristic
 function of the $G$-conjugacy class of $s$ is given by
 $$\psi=\frac{1}{|C|}
   \sum_{\theta\in\hat T}\theta(s^{-1})\,R_{T,\theta}$$
 where $\hat T:=\Irr(T)$.
\end{lem}

\begin{proof}
In the connected case, this is just \cite[Prop.~7.5.5]{Ca}. We mimic the
proof given there. Let $\psi'$ denote the characteristic function of $[s]$.
The claim follows if we can show that
$\langle\psi',\psi'\rangle=\langle\psi,\psi'\rangle=\langle\psi,\psi\rangle$.
Clearly, $\langle\psi',\psi'\rangle=1/|C|$. Next.
$$\langle\psi,\psi'\rangle
    =\frac{1}{|G|\,|C|}\sum_{\theta\in\hat T}\sum_{g\in[s]}
     \theta(s^{-1})\,R_{T,\theta}(g)
    =\frac{1}{|C|^2}\sum_{\theta\in\hat T}\theta(s^{-1})\,R_{T,\theta}(s).
$$
In our situation where $s$ is semisimple regular, contained in the unique
maximal ``torus'' $\bT$, the character formula \cite[Prop.~2.6]{DM94} for
$R_{T,\theta}$ takes the following form:
$$R_{T,\theta}(s)=\frac{1}{|T|\,|C_G^\circ(s)|}\sum_{\{h\in G\mid s\in T^h\}}
  |C_T^\circ(s)|\,\theta(s^h)=\frac{1}{|T|}\sum_{h\in N_G(T)}\theta(s^h).$$
So
$$\langle\psi,\psi'\rangle
  =\frac{1}{|C|^2\,|T|}\sum_{h\in N_G(T)}\sum_{\theta\in\hat T}
     \theta(s^{-1})\theta(s^h)
  =\frac{1}{|C|^2\,|T|}\sum_{\atop{h\in N_G(T)}{s^h\sim s\text{ in }T}}|\hat T|
  =\frac{1}{|C|}$$
by the orthogonality relations for $\hat T$. Finally
$$\langle\psi,\psi\rangle
  =\frac{1}{|C|^2}\sum_{\theta,\theta'\in \hat T}\theta(s^{-1})\theta'(s)\,
   \langle R_{T,\theta},R_{T,\theta'}\rangle.$$
But by \cite[Prop.~4.8]{DM94} we have
$$\langle R_{T,\theta},R_{T,\theta'}\rangle
  =\frac{1}{|T_0^\circ|}\,|\{g\in N_{G_0}(T)\mid {}^g\theta=\theta'\}|,$$
so
$$\langle\psi,\psi\rangle
  =\frac{1}{|C|^2\,|T_0|}\sum_{g\in N_{G_0}(T)}\sum_{\theta\in \hat T}
   \theta(s^{-1})\ {}^g\theta(s)
  =\frac{1}{|C|^2\,|T_0|}\sum_{\atop{g\in N_{G_0}(T)}{s^g\sim s}}|C|
  =\frac{1}{|C|}.$$
\end{proof}

Remember our standing assumptions that $\bG/\bG^\circ$ is cyclic and consists
of semisimple elements and $F$ acts trivially on $\bG/\bG^\circ$. We now give
a proof for an explicit bound on character values on regular semisimple
elements (the key argument is taken from Malle
\cite[\S5]{MaDis} where the connected case is treated):

\begin{thm}    \label{thm:bound-dis}
 Let $G=\bG^F$ as above, $s\in G^\circ\sigma$ regular semisimple, lying in
 the (unique) maximal ``torus'' $\bT$ of $\bG$, and $\chi\in\Irr(G)$.
 Then
 $$|\chi(s)|\le |W_G(\bT^\circ)|\le|W(\bG)|,$$
 where $W_G(\bT^\circ):=N_G(\bT^\circ)/T^\circ$ and
 $W(\bG):=N_\bG(\bT^\circ)/\bT^\circ$.
\end{thm}

\begin{proof}
Let $C:=C_G(s)$. If $\psi$ denotes the characteristic function of the
$G$-conjugacy class $[s]$ of $s$ then
$$\langle\chi,\psi\rangle=\frac{1}{|G|}\sum_{g\in G}\chi(g)\psi(g^{-1})=
  \chi(s)\,|[s]|/|G|=\chi(s)/|C|,$$
so $\chi(s)=|C|\langle\chi,\psi\rangle$. With Lemma~\ref{lem:charfun} this
gives
$$\chi(s)=|C|\langle\chi,\psi\rangle=\sum_{\theta\in\hat T}\theta(s^{-1})
  \langle\chi,R_{T,\theta}\rangle.$$
If $\langle\chi,R_{T,\theta}\rangle=0$ for all $\theta$ then $\chi(s)=0$
and the claim holds. So assume that there exists $\theta$ with
$\langle\chi,R_{T,\theta}\rangle\ne0$. Now, if $\theta'\in\Irr(T)$ is such
that $(\bT,\theta)$ is not geometrically $\bG$-conjugate to $(\bT,\theta')$,
then $R_{T,\theta},R_{T,\theta'}$ do not have any constituent in common by
Proposition~\ref{prop:disjoint}. Thus there is at most one
$N_G(\bT^\circ)$-orbit $\Theta(\chi)$ on $\hat T$ such that $\chi$ occurs in
$R_{T,\theta}$ for some $\theta\in\Theta(\chi)$. Also, as $\chi$ is
irreducible,
$$\langle\chi,R_{T,\theta}\rangle
  \le\langle R_{T,\theta},R_{T,\theta}\rangle^\half
  =|W_0(\bT,\theta)|^\half$$
(the last equality by \cite[Prop.~4.8]{DM94}), where
$$W_0(\bT,\theta):=\{w\in W_0(\bT)\mid\theta^w=\theta\},\qquad
  W_0(\bT):=N_{G_0}(\bT)/T_0^\circ.$$
Moreover, the orbit $\Theta(\chi)$ has length
$[N_G(\bT^\circ):N_G(\bT^\circ,\theta)]=[W_G(\bT^\circ):W_G(\bT^\circ,\theta)]$,
and $|\theta(s)|\le\theta(1)\le[T:T^\circ|^\half=[G:G^\circ]^\half$, so that
finally
$$\begin{aligned}
 |\chi(s)|&=|C|\cdot|\langle\chi,\psi\rangle|
  =|\sum_{\theta\in\Theta(\chi)}\theta(s)^{-1}\langle\chi,R_{T,\theta}\rangle|\\
 &\le \sum_{\theta\in\Theta(\chi)}|\theta(s)^{-1}||W_0(\bT,\theta)|^\half
  \le |\Theta(\chi)|\cdot[G:G^\circ]^\half|W_0(\bT,\theta)|^\half\\
 &=|W_G(\bT^\circ)|\cdot|W_G(\bT,\theta)|^\half/W_G(\bT^\circ,\theta)
  \le|W_G(\bT^\circ)|
\end{aligned}$$
for any $\theta\in\Theta(\chi)$.
\end{proof}

The second to last term in the previous inequality is even slightly better
than our claim whenever $\theta$ is not in general position (i.e., when
$\chi$ is not an irreducible Deligne-Lusztig character $\pm R_{T,\theta}$).

\begin{rem}
 Assume that $\bG=\bG^\circ$ is connected. 
\begin{enumerate}[\rm(a)]
 \item If $\chi$ is unipotent, so contained in some $R_{T,1}$, then
  $W(\bT,\theta)=W_G(\bT)$ and we obtain the bound
  $|\chi(s)|\le|W_G(\bT)|^\half$.
 \item Since there are roughly $|T|/|W_G(\bT)|$ characters not vanishing on
  a regular element $t\in T$, and $|C_G(t)|=|T|$, one might expect an average
  character value of $|W_G(\bT)|^\half$ on $t$. The example of $G=\SL_2(q)$
  shows that character values will be larger than this: there exist irreducible
  characters which on regular semisimple elements of order $q+1$ take value
  $\zeta+\zeta^{-1}$, $\zeta$ a  $q+1$st root of unity. This has absolute
  value arbitrarily close to $2=|W_G(\bT)|$.
 \item See also \cite[Thm.~3]{GLL} for an elementary proof of a result
  that still gives an explicit (but worse) bound.
\end{enumerate}
\end{rem}

For truly disconnected groups we expect that the correct upper bound
should be $|W_0(\bT)|$; this would follow from the above proof
once a stronger disjointness statement than the one in
Proposition~\ref{prop:disjoint} has been established, which
compares characters of $T$ instead of their restrictions to $T^\circ$.

%%%%%%%%%%%%%%%%%%%%%%%%%%%%%%%%%%%%%%%%%%%%%%%%%%%%%%%%%%%%%%%%%%%%%%%%%
\section{Algebraic Groups}       \label{sec:algebraic}

%%% char 0,  bigger fields

We now consider triples in simple algebraic groups. For
notational convenience, we now use standard font letters $G$ to denote
algebraic groups, and $G(q)$ for the fixed groups under Frobenius
endomorphisms with respect to an $\FF_q$-rational structure.

Fix a prime $p$ and let $k$ be the algebraic closure of $\FF_p$.
Let $G$ be a simple algebraic group over $k$. Of course, since $G$ is locally
finite, we cannot expect to generate $G$ with a finite number of elements. 
The replacement for generation is the property of generating the group
$G(q)$ over a finite field $\FF_q$ for arbitrarily large $q$. 

We do have to exclude $G=\SL_2$.  In this case,  there is a strong rigidity
result (see e.g. \cite{Mac}) and the conclusion of Theorem~\ref{thm:beaualg}
fails. 

We first need a result about maximal subgroups of simple algebraic groups. 

\begin{lem}   \label{lem:maxes}
 Let $G$ be a simple simply connected algebraic group over $k$.
 There exist positive integers $m$ and $m'$ (depending only upon the rank
 of the group) such that for any proper closed subgroup $H$ of $G$ one of
 the following holds:
 \begin{enumerate}
  \item $H$ is contained in a (maximal) positive dimensional subgroup;
 \item $H$ is not contained in any proper positive dimensional closed
   subgroup and $|H| \le m$; or
 \item $H$ contains $G(q)$ for some prime power $q > m'$ (including
 the possibility of a twisted form). 
 \end{enumerate}
 Moreover, there are only finitely many conjugacy classes of subgroups
 in (1) or (2).
\end{lem}

\begin{proof}
If $G$ is classical, this follows from Aschbacher's theorem on maximal
subgroups and representation theory. 
If $G$ is exceptional, then the result follows by the description of closed
maximal subgroups (and maximal Lie primitive groups) in Liebeck--Seitz
\cite{LS03}. The finiteness of the number of conjugacy classes follows by
\cite{LS03} for (1) and by Martin \cite[Prop. 1.4]{Mar03} for (2).
\end{proof}

We next define some subvarieties of $G^3$.
For conjugacy classes $C_1, C_2$ and $C_3$ in $G$ let
$$V(C_1, C_2, C_3)=\{(x_1, x_2, x_3)\mid x_i \in C_i, x_1x_2x_3=1\}.$$
For $M$ a  subgroup of $G$ we set
$V_M(C_1,C_2,C_3)=V(C_1,C_2,C_3)\cap M^3$ and
$$V(M)=\bigcup_{g\in G}\{(x,y,z)\in G^3\mid xyz=1,\langle x,y\rangle\le M^g\}.$$

We need a result about dimensions of centralizers. Recall that an algebraic
group is called reductive if it has no positive-dimensional closed connected
unipotent normal subgroup.

\begin{lem}  \label{lem:centralizers}
 Let $G$ be a reductive algebraic group with $G^\circ$ non-abelian.
 If $g \in G$, then $\dim C_G(g) > 0$.
\end{lem}

Note that a reductive group with $G^\circ$ non-abelian is necessarily
non-solvable, so the claim is just \cite[Cor.~10.12]{StEnd} (as was kindly
pointed out to us by T. Springer).

We also need the following fact (see \cite[Th.~1.2]{Gur07} for a closely
related result where $M$ is assumed to be reductive but $C$ is arbitrary):

\begin{lem}   \label{lem:classes}
 Let $G$ be an algebraic group, $M$ a closed subgroup, $C$ a semisimple
 conjugacy class of $G$. Then $C\cap M$ is a finite union of $M$-classes.
\end{lem}

\begin{proof}
It suffices to assume that $M/M^{\circ}$ is cyclic of order $d$ prime to
the characteristic (by considering the finitely many cosets $xM^{\circ}$ where
$x$ is semisimple).
By a result of Steinberg (see \cite[Prop.~1.3]{DM94}) every semisimple
element of $M$ normalizes a maximal torus $T$ of $M^{\circ}$. Since
all maximal tori of $M^{\circ}$ are conjugate,  every
element of $C \cap M$ is conjugate in $M$ to an element of $N_M(T)$.
So if we can prove that $C \cap N_M(T)$ is a finite union of $N_M(T)$-classes,
then clearly $C \cap M$ is a finite union of $M$-classes (indeed, at most the
number of $C \cap N_M(T)$ classes). Thus it suffices to assume that $T$ is
normal in $M^\circ$, whence $M^\circ=T\times U$ with $U$ unipotent. \par
Since $x$ is semisimple, $x^d \in T$. Thus, $D:=\langle T, x \rangle$ is a
complement to $U$ in $M$.   We claim that any two complements of
$U$ in $M$ are conjugate in $M$.   By induction on $\dim U$, it suffices
to assume that $U$ is abelian (and of prime exponent if the characteristic
is positive).   Clearly, $T$ is contained in any complement of $U$.  Thus,
it suffices to observe that $H^1(D/T, U)=0$ (since $U$ is a projective
$D/T$-module). As above, we may now assume that $M=D$ is reductive. Now we
can apply \cite[Th.~1.2]{Gur07} (or give a direct proof). 
\end{proof}

T. Springer has shown to us how this can also be obtained by a tangent
space argument (similar to an argument of Richardson \cite{Ri67}).

We note a trivial bound:
 
\begin{lem}  \label{estimate}
 Let $G$ be a simple algebraic group, $M\le G$ a closed proper subgroup.
 Let $C_i$, $1\le i\le 3$, be conjugacy classes of regular semisimple elements.
 \begin{enumerate}
%%  \item  $\dim V(C_l, C_2, C_3) \cap V(M) \le  \dim M+\dim G$.
  \item[\rm(a)] If $d$ is the minimal dimension of $C_M(x)$ for $x\in M$, then 
   $$\dim \left(V(C_1, C_2, C_3) \cap V(M)\right) \le  \dim M+\dim G - 2d.$$
  \item[\rm(b)]  In particular, if $M^\circ$ is nonabelian and reductive, then
   $$\dim \left(V(C_1,C_2,C_3)\cap V(M)\right) \le \dim M+\dim G - 2.$$
 \end{enumerate}
\end{lem}
 
\begin{proof}
Conjugation defines a surjective morphism $f: G\times W\rightarrow
V(C_1, C_2, C_3) \cap V(M)$, where $W= V_M(C_1,C_2,C_3)$. 

Let $x \in W$. Then $f(u^{-1}, x^u)=x$ for every $u \in M$.
Thus $\dim f^{-1}(x) \ge \dim M$ for every $x \in W$.
Since every element of the image of $f$ is in the $G$-orbit
of some element of $W$, each fiber of $f$  has dimension at least $\dim M$.

By Lemma~\ref{lem:classes} the intersection $C_i \cap M$ is a finite union
of $M$-classes. Thus, $\dim C_i \cap M \le \dim M - d$. It follows that
$\dim W \le 2(\dim M -d)$ and so 
$$\dim V(C_1, C_2, C_3) \cap V(M) 
  \le\dim G+\dim W-\dim M\le  \dim G+\dim M-2d$$
as claimed in~(a).

If $M^\circ$ is nonabelian and reductive, then every centralizer has positive
dimension by Lemma~\ref{lem:centralizers}.
Thus, $d \ge 1$ and (b) follows from (a).
\end{proof}
 
We can now  show that  $V(C_1, C_2, C_3)$
is irreducible of dimension $2 \dim G  - 3r$.    We thank Tonny Springer
for pointing out the fact that any irreducible component of this variety
has dimension at least that.
 
\begin{thm} \label{thm:dimtriples}
 Let $G$ be a simple simply connected algebraic group of rank $r$. 
 Let $C_i$, $1 \le i \le 3$, be classes of regular semisimple elements of $G$.
 Then  $V(C_1, C_2, C_3)$ is an irreducible variety
 of dimension  $2\dim G - 3 r$.
\end{thm}
 
\begin{proof}
We first  observe that every irreducible component of $V$ has dimension
at least  $2 \dim G - 3r$.  Let $W = G \times G \times G$,
$X=C_1 \times C_2 \times C_3 \subset W$ and $Y=\{(x,y,z) \in W \mid xyz=1\}$.
Then $\dim X = 3 \dim G - 3r$ and $\dim Y = 2 \dim G$. Since $X$ and $W$
are irreducible and $W$ is smooth and irreducible, it follows by
\cite[p. 146]{Weil} that indeed every component of $V= X \cap Y$ has
dimension at least $\dim X + \dim Y  - \dim W = 2 \dim G - 3r$.

If $r=1$, it is straightforward to compute directly (cf.~\cite{Mac}). Now
assume that $r > 1$.
 
Choose a large power $q$ of $p$ so that the $C_i$ are defined
over $G(q)$. Now we count the $\FF_q$-points of $V(C_1,C_2,C_3)$ using the
character formula  given in the introduction. 
We just use the following facts: 

\begin{enumerate}
 \item[\rm{(a)}] $|\chi(x)| \le c$ for some constant $c$ depending only on
  the rank of the group for any irreducible character $\chi$ of $G(q)$ and
  any regular semisimple element $x\in G(q)$, by Theorem~\ref{thm:bound-dis};
 \item[\rm{(b)}] $\sum_{\chi}  \chi(1)^{-1} \le 1 + O(q^{-1/2})$
  where the sum is over all irreducible characters of $G(q)$
  (this follows easily from Deligne-Lusztig theory as well, see
  the proof of \cite[Thm.~1.1]{LiSh2}, using that the Coxeter number of
  a simple algebraic group not of type $A_1$ is larger than~2);  and
\item[\rm{(c)}] $C_i(q)$ is a single $G(q)$ conjugacy class (since
  centralizers of semisimple elements in groups of simply connected type are
  connected, see for example \cite[Thm.~14.16]{MT10}).
\end{enumerate}
 
It follows that 
$$
|V(C_1,C_2,C_3)(q)|  = \frac{|G(q)|^2}{c_1c_2c_3} (1 +O(q^{-1/2})) = 
q^{2 \dim G - 3r}(1 + O(q^{-1/2})),
$$
where $c_i$ are the orders of centralizers in $G(q)$ of elements from $C_i$.

The Lang-Weil theorem \cite{LW} on the number of points of an irreducible
variety over a finite field now shows that $V(C_1,C_2,C_3)$ is irreducible
and of dimension as claimed.
\end{proof}

If $G$ is not simply connected, a variant of the previous result is still true.

\begin{cor}   \label{cor:notsc}
 Let $G$ be a simple algebraic group of rank $r$ and $C_i$, $1 \le i \le 3$,
 classes of regular semisimple elements of $G$. Then:
 \begin{enumerate}
  \item[\rm(a)] Every irreducible component of $V(C_1, C_2, C_3)$ has
   dimension $2\dim G - 3 r$.
  \item[\rm(b)] For $\pi:\hat{G}\rightarrow G$ a simply connected covering of
   $G$, with (finite, central) kernel~$Z$, choose conjugacy classes $D_i$,
   $1\le i\le 3$, of $\hat{G}$ that project onto $C_i$. Let
   $Z_i=\{z\in Z \mid zD_i = D_i\}$ and set $Z_0 = Z_1Z_2Z_3$. Then the
   number of components of $V(C_1,C_2,C_3)$ is $[Z:Z_0]$.
 \end{enumerate}
\end{cor}
 
\begin{proof}
For $z \in Z$, let $V(z)$ be the variety of triples
in $D_1 \times D_2 \times D_3$ with product~$z$. By the previous result,
this is an irreducible variety of dimension $2 \dim G -3r$ (this variety
is isomorphic to $V(D_1, D_2, z^{-1}D_3)$).
Let $X:=X(D_1, D_2, D_3) = \cup_z V(z)$.
Observe that $\pi(X)=V(C_1, C_2, C_3)$.
 
Suppose that $(x_1, x_2, x_3) \in V(C_1, C_2, C_3)$. Choose $y_i \in D_i$
with $\pi(y_i)=x_i$. Thus, $(y_1, y_2, y_3) \in V(z)$ for some $z \in Z$.
Indeed, we see that $\pi(V(z_1))=\pi(V(z_2))$ if and only if $z_1Z_0=z_2Z_0$,
and that the $\pi(V(z_i))$ are disjoint if $z_1Z_0 \ne z_2Z_0$.
Let $R$ be a set of coset representative for $Z/Z_0$.
Thus, $V(C_1, C_2, C_3)$ is a disjoint union of the $\pi(V(z))$,
$z \in R$.  Since $\pi$ is a finite map, this implies by the previous result
that  $\dim V(D_1, D_2, D_3) = \dim V(C_1, C_2, C_3)$ and that there
are $[Z:Z_0]$ different components.
\end{proof}

We  next consider connected but not necessarily reductive
groups $H$. Recall the notion of regular semisimple elements from
Definition~\ref{def:regular}.

\begin {cor}   \label{cor:upperbound}
 Let $H$ be a connected algebraic group, with unipotent radical $U$ and set
 $L=H/U$, a reductive group. Let $C_i$, $1\le i\le 3$, be regular semisimple
 conjugacy classes of $H$.    Let $V=V(C_1, C_2, C_3)$ be nonempty.
 Then every irreducible component of $V$ has dimension  $2 \dim [H,H]-3r$
 where $r$ is the rank of $[H,H]/U \cong [L,L]$.
\end{cor}

\begin{proof} 
We first claim that $\dim C_H(x)=r$ for $x \in C_i$.
Since $x$ is regular semisimple,  we see that $C_U(x)=1$ and
$C_H(x) \cong C_{H/U}(x)$.  Since $H/U$ is reductive, the claim follows.

Now argue (using the bound for the dimension of a component of an
intersection) as in the proof of Theorem~\ref{thm:dimtriples} to deduce that
every irreducible component of $V$ has dimension at most $2 \dim [H,H] - 3r$.

We now prove the reverse inequality. If $U=1$, then the result is clear by
the result for simple groups.  Suppose that $x_i \in C_i$
with $x_1x_2x_3 \in U$.  Then we can choose $u_i \in U$ with $u_1$ and
$u_2$ arbitrary and $u_3$ uniquely determined by $u_1, u_2$ such
that $\prod_i (x_iu_i)=1$.  Since $x_i$ is semisimple regular, 
$x_iu_i \in C_i$.  Thus, the result follows by the reductive case. 
\end{proof} 

There is a version of this for some disconnected groups as well.
We will say an algebraic group is \emph{almost simple} if its connected
component is simple and contains its centralizer. The proof of the next
result is identical to that of the upper bound in Theorem~\ref{thm:dimtriples}
using Theorem~\ref{thm:bound-dis}.

\begin{cor} \label{cor:disconnected}
 Let $G$ be a not necessarily connected almost simple algebraic group.
 Let $C_i$, $1\le i\le 3$, be conjugacy classes of regular semisimple elements
 of $G$. Then $\dim V(C_1,C_2,C_3)\le 2\dim G - \sum_{i=1}^3(\dim G-\dim C_i)$.
\end{cor}

Note that the result applies equally well to the case that
$G$ is (essentially) a direct product of such groups.

We only need to apply the corollary in a special case which depends upon
the following result.

\begin{lem}   \label{lem:Dcentralizers}
 Let $G$ be a disconnected algebraic group such that $G^{\circ}$ is simple
 and $G/G^{\circ}$ is generated by a graph automorphism $\tau$ of order $d$.
 Let $X=\tau G^{\circ}$. Then the minimum dimension of $C_G(x)$, $x \in X$,
 is the number of orbits of $\tau$ on the Dynkin diagram of $G$. Moreover,
 a generic element of $X$ is semisimple regular.
\end{lem}

\begin{proof}
By inspection, we can choose $\tau$ such that $C_G(\tau)^{\circ}$ is a simple
algebraic group of rank equal to the number of orbits of $\tau$ on the Dynkin
diagram, and $\tau$ centralizes a regular semisimple element $g\in T$ of order
prime to $d$ in some $\tau$-stable maximal torus $T\le G^{\circ}$.
%%%  SL_{2n}, Sp_{2n},  SL_{2n+1}, SO_{2n+1},  D_n, B_n,  D_4, G_2,  E_6, F_4 
Thus,  $(\tau g)^d = g^d$ is regular semisimple in $G^{\circ}$.  It follows
that $(\tau g')^d$ is regular semisimple for an open subvariety of $g' \in T$
(and so also for $g' \in G^{\circ}$), whence the second statement follows.
\par

Let $f:G^{\circ} \times \tau T \rightarrow X$ be the conjugation map.
Let $t \in T_0$ with $(\tau t)^d$ semisimple regular. Now for $t' \in T$, if
$f(g, \tau t')=g(\tau t')g^{-1} = \tau t$, then $\tau t'$ and $\tau t$ have
the same centralizer in $G^0$  (namely $T_0$) and so $g$ normalizes $T_0$
and so $T$. Thus, the generic fiber of $f$ has dimension equal to $\dim T$
and $f$ is a dominant map. This shows that for elements
in a nonempty subvariety of $X$,  the dimension of the centralizer is equal
to that of $T_0$, whence that is the minimum dimension. This completes
the proof.
\end{proof}

We will also need the  following result which is  a version of
\cite[Thm.~3.3]{Gur98}. 

\begin{thm}  \label{thm:open}
 Let $G$ be a simple algebraic group over an algebraically closed field $k$
 of characteristic $p \ge 0$. Let $N$ be a positive integer.
 \begin{enumerate}
  \item[\rm(a)] If $p=0$, then
   $\{(x,y)\in G \times G\mid G=\overline{\langle x,y\rangle}\}$
   is a nonempty open subvariety of $G \times G$.
  \item[\rm(b)] If $p>0$, then $\{(x,y)\in G \times G\mid G(q)^g\le
   \overline{\langle x, y\rangle} \text{ for some } g \in G, q > N\}$ is a
   nonempty open subvariety of $G \times G$.
 \end{enumerate}
\end{thm}

\begin{proof}
We first give the proof for $p > 0$.  It suffices to assume that
$k=\overline{\FF_p}$.  By Guralnick--Tiep \cite[Thm.~11.7]{GT08} for the
classical groups and using Liebeck-Seitz \cite{LS04} for the exceptional
groups, there is a finite collection of rational $kG$-modules such that the
only proper closed subgroups of $G$ irreducible on all of them are conjugates
of $G(q)$ for some sufficiently large $q$ (this includes twisted versions).
The set of pairs which are reducible on any finite
collection of modules is a closed condition. Since $G(q)$ is $2$-generated,
the set is nonempty (one only needs to know this for some sufficiently
large $q$).

If $p=0$, then no proper closed subgroup is irreducible on the collection
of submodules given. Since one can easily find two elements which
generate a dense subgroup, the result follows.
\end{proof}

\begin{thm}  \label{thm:beaualg}
 Assume that $G$ is a simple simply connected algebraic group of rank $r>1$
 over the algebraic closure $k$ of $\FF_p$. Let $C_i$, $1 \le i \le 3$,
 be regular semisimple classes and $V=V(C_1, C_2, C_3)$. Assume also that
 $G \ne \SL_3$ if each $C_i$ consists of elements of order $3$ modulo $Z(G)$.
 Fix a positive integer $m$. Then for a generic $x:=(x_1, x_2, x_3) \in V$
 we have that $\langle x_1, x_2 \rangle \ge G(q)$ for some $q > m$.
\end{thm}
 
\begin{proof}
As noted above, the set of pairs $(u,v) \in G \times G$
such that $\langle u,v \rangle $ contains $G(q)$ for some $q > m$ is open
in $G \times G$. Thus, if the result fails, it follows by Lemma~\ref{lem:maxes}
that for every $x \in V$, $\langle x_1, x_2 \rangle \le M$ for some closed
subgroup $M$ either with $M$ a positive dimensional maximal closed subgroup
of $G$ or with $|M| \le m$ and $M$ not contained in any proper positive
dimensional subgroup. There are only finitely many conjugacy classes of
such $M$ as we have already noted in Lemma~\ref{lem:maxes}.
Thus, $V$ is contained in the finite union of these $V(M)$.
Since $V$ is irreducible by Theorem~\ref{thm:dimtriples}, this implies that
$V$ is contained in the closure of $V(M)$ for some fixed $M$.
As we have seen above, this implies that
$2 \dim G - 3r = \dim V\le \dim M + \dim G $, whence
$\dim M \ge \dim G - 3r$. In particular, since $r > 1$
this implies that $\dim M > 0$, and even that $\dim M > r$
unless possibly $G=\SL_3$ where $r=2$.

Suppose that $M$ is a maximal subgroup that is connected (or more generally
$M^\circ$ contains all semisimple elements of $M$). It follows by
Corollary \ref{cor:upperbound} that $\dim V_M(C_1, C_2, C_3)
\le 2\dim [M,M] - 3r_1$, where $r_1=\rk([M,M])$. Arguing as
in the proof of Lemma~\ref{estimate}, this implies that
$$\begin{aligned}
2\dim G -3r = \dim V(C_1,C_2,C_3)\le& (\dim G-\dim M) + \dim V_M(C_1,C_2,C_3)\\
  \le& \dim G  - \dim M + 2 \dim [M,M] - 3r_1.
\end{aligned}$$
Thus $\dim G \le 2 \dim [M,M] - \dim M -3r_1 + 3r$. This clearly
cannot be the case if $r_1 =r$. In particular, $M$ cannot be of type $D_n$
in $C_n$ in characteristic $2$. If $r_1 \le r-1$, this implies that
$\dim G  \le 2 \dim [M,M] - \dim M +3 $. If $M$ is not semisimple
and $r_1=r-1$, this yields $\dim G \le \dim M +1$ which cannot
occur since
$r > 1$.   If $M$ is semisimple of rank $r-1$, the inequality above gives
$\dim G \le \dim M + 3$. There are no proper semisimple subgroups of
codimension at most $3$. Thus, we see that either $r_1 < r-1$ or at least
one of the classes $C_i \cap M$ is not contained in $M^{\circ}$.
In particular, $M$ is reductive.

Now $C_i \cap M$ is a finite union of conjugacy classes of $M$ by
Lemma~\ref{lem:classes}. Thus,
$$\dim V(C_1,C_2,C_3)\cap M^3=\dim V(D_1,D_2,D_3)$$
where $D_i$ is a conjugacy class of $M$ with $D_i \subseteq C_i \cap M$.

Suppose that $M^\circ$ is a torus.  As we have noted above, this implies
that $r =2$ and $G=\SL_3$.   Indeed, arguing as above, we see that
$\dim V(D_1, D_2, D_3) +  6  \ge \dim V(C_1, C_2, C_3) = 10$
or $\dim V(D_1, D_2, D_3) \ge 4$.   Clearly,  $\dim V(D_1, D_2, D_3)
\le \dim D_1 + \dim D_2$.
It follows that the $D_i$ have finite centralizer in $M^{\circ}$.
The only possibility is that the $D_i$ consist of elements of order $3$,
but this is excluded by hypothesis.

Thus, every element of $M$ has a positive dimensional centralizer in $M$
by Lemma \ref{lem:centralizers}. Then with Lemma~\ref{estimate}(b) the
argument above gives a bit more:
$$\dim M \ge \dim G - 3r + 2.$$

Consider the case that $G$ is classical.
First suppose that $M$ is reducible on the natural module. 
Then $M$ must be the stabilizer of a nondegenerate space of dimension less
than $1/2$ the dimension of the space and $\dim M \ge \dim G - 3r +2$.
The only possibility is that $G$ is an orthogonal group,
$M$ is the stabilizer of a nondegenerate $1$-space and $p > 2$ (if $p=2$,
every semisimple element in $M$ is contained in $M^\circ$ and so the better
inequality applies).
 
If $G$ is of type $D_n$, $n \ge 4$, then  $M = 2\times M^\circ$ and so we see
that $\dim (V \cap M^3) \le 2 \dim M^\circ - 3(n-1)$.  Thus,
arguing as above,  $\dim M \ge \dim G -2$, a contradiction.

Suppose that $G$ is of type $B_n, n \ge 3$ and $M^\circ=D_n$.  Then
apply Corollary \ref{cor:disconnected} and Lemma \ref{lem:Dcentralizers}
and argue as above. This gives the inequality 
$$2 \dim G - 3n \le \dim G - \dim M  + 2 \dim M - 3(n-1)$$
or $\dim G \le \dim M + 3$, a contradiction.
 
If $M$ is irreducible but not almost simple, then either the natural module
is imprimitive, tensor decomposable or tensor induced. By inspection,
the only example 
with $\dim M \ge \dim G - 3r + 2$ is for $G=\Sp_4$ and $M=\SL_2 \wr 2$.
If $xyz=1$ with $x,y,z \in M$,
then at most two of the elements can live outside $M^\circ$. It follows
that at least one of the three elements has a $2$-dimensional centralizer
in $M^\circ$, whence the argument above gives
$\dim V(C_1, C_2, C_3) \cap M^3 \le \dim M - 3$ and we obtain a contradiction.

Suppose that $M$ is almost simple. 
Using the  bound $\dim M \ge \dim G - 3r + 2$ eliminates almost all
possibilities (by the results of L\"ubeck \cite{Lu01} where he explicitly
computes all irreducible modules of dimension less than $r_M^3/8$). The
exceptions are the cases where $(M,G)=(\SL_2=\SO_3,\SL_3)$, $(M,G)=(G_2,B_3)$
or $(M,G)=(\Sp_n,\SL_n)$ with $n=4,6$. In all these cases, 
$M$ is connected, so it follows that $\dim V(M) \le \dim G + \dim M - 3r_1$
and this is sufficient to show that $\dim V(M) < \dim V$.
 
Now let $G$ be an exceptional group. By \cite{LS03},
it follows that $\dim M  < \dim G - 3r +2$ for any maximal
reductive subgroup of $G$, whence the result follows.
\end{proof}
 
\begin{rem}
In the excluded case $G=\SL_3$ and $C_i$ containing elements of order~3,
the proof does not exclude that we might generate a subgroup of the normalizer
of the torus. Indeed that's what happens since it is well-known (and easy to
see) that the triangle group generated by three elements of order~3 is solvable.
\end{rem}

Another application of the Lang-Weil theorem \cite{LW}  gives:
 
\begin{cor}  \label{cor:beaulie}
 Let $G$ and the $C_i$ be as in Theorem~\ref{thm:beaualg}.
 If the $C_i$ are defined over $\FF_q$ with $q$ sufficiently large, then there
 exist $x_i \in C_i$ with product $1$ which generate $G(q)$ 
\end{cor}

\begin{proof}
By the previous results, we just have to count the number of triples in $V$
which are conjugate to a triple in some subfield group  $G(q_0)$.
It is easy to see that these do not contribute enough to affect the result.
\end{proof}

\begin{rem}   \label{rem:larsen}
It follows by an easy argument that the results extend to arbitrary fields
(and so we obtain Theorem \ref{thm:alggroup triples}). 
Here is the sketch (we thank Michael Larsen for pointing this out to us).

Let $G$ be a simple simply connected algebraic group over an algebraically
closed field $k$. Let $C_i$, $1 \le i \le 3$, be semisimple regular conjugacy
classes of $G$.  Let $V=V(C_1,C_2,C_3)$ be the variety of triples
$(x_1, x_2, x_3)$ with
$x_i \in C_i$ and product $1$. Since the $C_i$ are semisimple classes this is
a closed subvariety of $G \times G \times G$.  Note that the argument
given in the proof of Theorem \ref{thm:dimtriples} shows that every
irreducible component of $V$ has dimension at least $2 \dim G - 3r$.

Note that $V$ is  defined over some finitely
generated subring $R$ of $k$. If $M$ is a generic maximal ideal of $R$, then
the reduction of $V(R)$ modulo $M$ will have the same dimension as $V$ and the
same number of irreducible components  over the algebraic
closure of $R/M$. Since $R/M$ is a
finite field, the result now follows from Theorem~\ref{thm:dimtriples}. 
\end{rem}

The proof of Theorem \ref{thm:beaualg} now goes through verbatim
and so holds as stated for $k$ of positive characteristic with $C_1, C_2, C_3$
torsion classes of regular semsimple elements.    

If the characteristic is $0$ or one of the $C_i$ consists of classes of
infinite order, the proof shows:

\begin{thm}   \label{thm:triplegen2}
 Assume that $G$ is a simple simply connected algebraic group of rank at
 least $2$ over an algebraically closed field $k$ of characteristic $p \ge 0$.
 Let $C_i, 1 \le i \le 3$ be semisimple regular conjugacy classes of $G$.
 Assume either that $p=0$ or one of the classes $C_i$ consists of elements
 of infinite order. If $(x_1, x_2, x_3)$ is a generic triple in
 $V(C_1, C_2, C_3)$, then $\langle x_1, x_2 \rangle$ is Zariski dense in $G$. 
\end{thm}

In characteristic $0$, it follows that the set of triples in this variety that
generate a dense subgroup contains a nonempty open subvariety. 

%%The sufficiently large depends on the choice of the $C_i$. By being a bit
%%more careful, one sees that it really only depends on the rank of $G$ (and
%%not even on the characteristic).

%%% char 0, elements of infinite order in char p
%%% increasing rank, uniformity

We close this section by demonstrating how the result of
Theorem~\ref{thm:alggroup triples} can be extended by using asymptotic
estimates on character values:

\begin{thm}  \label{thm:moretriples} 
 Let $G$ be a simple, simply connected algebraic group of exceptional type.
 Let $C_1$ and $C_2$ be conjugacy classes of regular semisimple elements in $G$.
 Let $C_3$ be a conjugacy class of $G$ such that $x_3 \in C_3$ has centralizer
 dimension $\dim C_G(x_3) \le d_G$ with $d_G$ as in Table~\ref{tab:dG}. Then
 the closure of $V(C_1,C_2,C_3)$ is irreducible of
 dimension $2(\dim G - \rk(G))  - \dim C_G(x_3)$.
\end{thm}

\begin{table}[htbp] 
  \caption{Bounds for centralizer dimensions and character degree polynomials}
  \label{tab:dG}
\[\begin{array}{|l|ccccc|} 
\hline
 G& G_2& F_4& E_6& E_7& E_8\\
\skipa \hline \hline
 d_G& 5& 21& 19& 39& 97\\
 e_G& 5& 15& 16& 27& 57\\ \hline
\end{array}\]
\end{table}

\begin{proof}
Let $V$ denote the closure of $V(C_1,C_2,C_3)$.
First assume that we are over the algebraic closure of a finite field. 
For $q$ a prime power, let $G(q)$ denote the group of fixed points of $G$
under a standard Frobenius endomorphism $F$ of $G$ corresponding to an
$\FF_q$-rational structure on $G$.   Choose $q$ such that the $C_i$ are all
defined over $\FF_q$, $x_3 \in G(q)$ and that each component of $C_G(x_3)$
is invariant under $F$.   Let $x_i \in C_i(q)$ for $i=1,2$. 
\par
We claim that $|V(C_1,C_2,C_3)(q)|= q^m(1 + o(1))$, where
$m=2(\dim G - \rk(G))  - \dim C_G(x_3)$. 
\par
By Lang's Theorem  $C_3$ splits into $e$ classes in $G(q)$ where $e$ is the
number of conjugacy classes of $A(x_3):=C_G(x_3)/C_G(x_3)^{\circ}$ (since $F$
acts trivially on the set of components, see e.g. \cite[Thm.~21.11]{MT10}).
Let $u_1, \ldots, u_e$ be representatives for the conjugacy classes of
$A(x_3)$. Let $d = \dim C_G(x_3)$. Then the sizes of the $G(q)$-conjugacy
classes in $x^G \cap G(q)$ will be $a_i^{-1} q^{\dim G-d} + O(q^{\dim G-d-1})$
($1\le i\le e$), where
$a_i$ is the order of the centralizer of $u_i$. Note that $\sum a_i^{-1} = 1$
(since  $A(x_3)$ is the disjoint union of its conjugacy classes). 
\par
Thus,
$$
|V(C_1, C_2, C_3)(q)| = \frac{|C_1(q)||C_2(q)|}{|G(q)|^2} 
 \sum_{i=1}^e |u_i^{G(q)}| \sum_{\chi \in \Irr(G)}  \frac{\chi(x_1)\chi(x_2)\chi(u_i)}{\chi(1)}.
$$
Since $\sum a_i^{-1}=1$,  the contribution from the trivial character is $q^m + O(q^{m-1})$. We argue as in the proof of Theorem~\ref{thm:dimtriples} and
show that $\sum_{\chi \ne 1}  |\chi(u_i)| \chi(1)^{-1}\le  o(1)$. Note that
by the result of Gluck \cite{Gluck} $|\chi(u_i)|/\chi(1)=o(1)$ for all
$1\ne\chi\in\Irr(G(q))$ and $1\ne x\in G(q)$, so in the above sum we may
ignore any bounded number of non-trivial characters. 
\par
By Lusztig's Jordan decomposition of characters, $\Irr(G)$ is the disjoint
union of Lusztig series, indexed by semisimple classes in the dual group
$G^*(q)$, each of size bounded only in terms of $r=\rk(G)$.
Since $G^*(q)$ has at most $c_1q^r$ semisimple conjugacy classes (see
\cite[Thm. 26.10]{MT10}), for some
$c_1>0$, we conclude that $|\Irr(G(q)|\le c_2 q^r$. By the orthogonality
relations we have $|\chi(x_3)|\le c_3q^{d_G/2}$ for all $\chi\in\Irr(G(q))$.
The smallest character degree of $G(q)$ not lying in the Lusztig series of
an isolated element is of the form $c_4q^{e_G}$ with $e_G$ as given in
Table~\ref{tab:dG}, for example by \cite{Lu01b}.
The claim now follows by Theorem~\ref{thm:bound-dis} since $d_G<2(e_G-r)$.
\par
Since $V(C_1, C_2, C_3)$ is open in $V$ (as $C_3$ is open in its closure, see
e.g. \cite[Prop. 5.4]{MT10}),
it follows that $|V(q)| = |V(C_1, C_2, C_3)(q)|+O(q^{m-1}) = q^m(1 + o(1))$. 
We complete the proof arguing as we did in Theorem \ref{thm:dimtriples}. By
Lang--Weil \cite{LW} $V$  has exactly one irreducible component with
dimension $m$ (and none of dimension greater than $m$).
By \cite[p.~146]{Weil},  every irreducible component  has dimension at least
$m$ and so the variety is irreducible.
\par
The result for arbitrary algebraically closed  fields follows as in
Remark~\ref{rem:larsen}.
\end{proof}

If $C_3$ is a semisimple class, the proof is a bit easier as
$V(C_1, C_2, C_3)$ is closed and $C_G(x_3)$ is connected. 
Note that the only non-trivial semisimple elements in $E_8$ with centralizer
dimension larger than~$97$ are involutions with centralizer of type $D_8$
or $E_7\times A_1$, and elements in a 1-dimensional torus $T_1$  with centralizer
$T_1E_7$.

%%%%%%%%%%%%%%%%%%%%%%%%%%%%%%%%%%%%%%%%%%%%%%%%%%%%%%%%%%%%%%%%%%%%%%%%%
\section{Generating Conjugacy Classes}   \label{sec:genclass}

We now return to finite groups.
In the proof of Theorem \ref{thm:beau}, we showed that in many finite simple
groups there exist conjugacy classes $C$ and $D$ such that
$G$ is generated by any pair of elements in $C \times D$. Moreover, in almost
all the cases $C$ and $D$ were $\Aut(G)$-invariant. We investigate this
further in this section. We will prove the following version of
Theorem~\ref{thm:gen classes}

\begin{thm}  \label{thm:strong gen classes}
 Let $G$ be a finite almost simple group with socle $S$. There exist
 conjugacy classes $C$ and $D$ of $G$ such that $S \le \langle c, d \rangle$
 for all $(c,d) \in C \times D$. Moreover aside, from the cases
 $S=\OO^+_8(q)$, $q \le 3$, we may take $C,D \subset S$.  In all cases,
 we may assume that $D$ is contained in $S$. 
\end{thm}

If $G$ contains a triality automorphism of $S=\OO^+_8(2)$, the classes
$C$ and $D$ cannot both be chosen to be contained in $S$.
It seems likely this is also true for $q=3$.

In particular, a special case of the result is the following:

\begin{cor}   \label{cor:strong gen classes} 
 Let $G$ be a finite simple group other than $\OO_8^+(q)$, $q \le 3$.
 There exist subsets $C, D$ of $G$ each invariant under $\Aut(G)$
 such that $S = \langle c, d \rangle$ for each $(c,d) \in C \times D$.
\end{cor}

One way of producing such classes is rather obvious:

\begin{lem}   \label{lem:unique}
 Let $G$ be a finite group with $g \in G$. Assume that $g$ is contained in a
 unique maximal subgroup $M$ of $G$. Let $C=g^G$. Let $D$ be the set of
 derangements of $G$ in the permutation action on $G/M$.
 Then $G = \langle g, h \rangle$ for any $h \in D$.
\end{lem}

Of course, $D$ is always nonempty (by the well-known observation of Jordan).
Moreover, if the $G$-class of $M$ is $\Aut(G)$-invariant, then we can take
$C=g^{\Aut(G)}$ and $D$ will also be $\Aut(G)$-invariant. Such elements exist
in many (but not all) finite simple groups.

We start the proof of Theorem~\ref{thm:strong gen classes} by showing that
alternating groups satisfy the result.

\begin{prop}   \label{prop:altgen}
 Let $G=\fA_n$, $n \ge 5$. There exist conjugacy classes $C,D$ of $\fS_n$
 contained in $\fA_n$ such that $G = \langle x, y \rangle$ for any $x \in C,
 y \in D$.
\end{prop}

\begin{proof}
First suppose that $n=4m\ge 8$ is divisible by~4. Let $g$ be a product of
a $2m+1$ cycle and a $2m-1$ cycle. Thus, some power of $g$ is a $2m-1$ cycle.
Let $M$ be a maximal subgroup containing $g$. We show that $M$ must be
intransitive. Since there is a unique intransitive such maximal subgroup, the 
result will follow by Lemma~\ref{lem:unique}.

Since $\gcd(2m+1, 2m-1)=1$,  $M$ cannot be transitive and imprimitive.
If $M$ is primitive, it follows by Williamson \cite{Wi73} that $M$ cannot exist.

Next suppose that $n = 2m \ge 10$ with $m$ odd.  Let $g$ be a product
of disjoint cycles of lengths $m-2$ and $m+2$.  Argue precisely as above.

If $n=6$,  let $C$ be the set of all $5$-cycles and let $D$ be the class
of elements of order $4$ in $G$.  The only maximal subgroups containing
an element $g$  of order $5$ are $\fA_5$ (two classes), none of which
contain an element of order $4$  (thus, $C$ and $D$ are $\Aut(G)$-invariant).

Now suppose that $n \ge 5$ is odd. Let $g$ be an $n$-cycle. If $n$ is prime, 
let $h$ be a  $3$-cycle.  Then  $G=\langle g, h \rangle$ (since any primitive
group containing a $3$-cycle contains $G$).  If $n$ is not prime, let $q$ be
a prime with $n/2 < q < n-3$, which exists by Bertrand's postulate. Let $h$
be the product of a $3$-cycle and a $q$-cycle.  Then $\langle g,h \rangle$
is clearly primitive and contains a $3$-cycle, whence the claim follows.
\end{proof}

\begin{prop}   \label{prop:sporgen}
 Let $G$ be a sporadic simple group. Then there exist $\Aut(G)$-classes $C,D$
 of $G$ such that $G=\langle x, y \rangle$ for any $(x,y)\in C \times D$.
\end{prop}

\begin{proof}
This follows by the proof of Proposition \ref{prop:spor}, except for $J_2$
and the Tits group. For the first group, all pairs $(x,y)\in 5c\times 7a$
generate, for $\tw2F_4(2)'$ the same is true for all pairs
of elements from $13a\times 16a$, by \cite{Atl}.
\end{proof}

It remains to consider the finite simple groups of Lie type.

\begin{prop}   \label{prop:genL}
 Theorem~\ref{thm:strong gen classes} holds for the simple linear groups
 $\PSL_n(q)$.
\end{prop}

\begin{proof}
First consider $\PSL_2(q)$. The alternating groups
$\PSL_2(4)\cong\PSL_2(5)\cong\fA_5$ and $\PSL_2(9)\cong\fA_6$ were treated in
Proposition~\ref{prop:altgen}. For $q\ge11$
let $C_1$ contain elements of order~$(q+1)/d$ and $C_2$ elements of order
$(q-1)/d$, where $d=\gcd(2,q-1)$. Then by the well-known classification of
subgroups of $\PSL_2(q)$, any pair $(x,y)\in C_1\times C_2$ will generate.
For $\PSL_2(7)$ we let $C_2$ be a class of 7-elements instead.
\par
For $n\ge3$ let $C_1$ contain elements of order $(q^n-1)/(q-1)/d$, and $C_2$
elements of order $(q^{n-1}-1)/d$, where $d=\gcd(n,q-1)$. If $n\ge5$ then
any pair from $C_1\times C_2$ generates by our Corollary~\ref{cor:twozsig2},
unless $(n,q)\in\{(6,2),(7,2)\}$ when one of the two Zsigmondy primes does
not exist. In the first of these cases, the only proper overgroup of elements
from $C_2$ is an end-node parabolic, while in the second the only proper
overgroup of elements from $C_1$ is the normalizer of a Singer cycle, by
\cite[Lemmas~2.3 and~2.4]{GM10}, but neither contains elements of order
$2^6-1$. \par
Similarly, for $3\le n\le4$ it follows from \cite[Lemma~2.3]{GM10} that any
pair generates unless possibly $(n,q)\in\{(3,2),(3,4),(4,2),(4,3)\}$. The
groups $\PSL_3(2)\cong\PSL_2(7)$ and $\PSL_4(2)\cong\fA_8$ were already
considered before. According to \cite{Atl} the group $\PSL_3(4)$ is generated
by any pair of elements of orders~5 and~7; the group $\PSL_4(3)$ is generated
by any pair of elements of orders~5 and~13.
\end{proof}

\begin{prop}   \label{prop:genU}
 Theorem~\ref{thm:strong gen classes} holds for the simple unitary groups
 $\PSU_n(q)$, $n\ge3$.
\end{prop}

\begin{proof}
For $n\ge8$ this was already shown in Proposition~\ref{prop:U}, using the two
classes in Table~\ref{tab:classtorus}. For $n\le7$ odd let $C_1$ contain
elements of order $(q^n+1)/(q+1)/d$, where $d=\gcd(n,q+1)$. By
\cite[Lemma 2.5]{GM10} the only maximal overgroup of such an element is the
normalizer of the maximal torus of that order, in which case we're done by
Lemma~\ref{lem:unique}, or $(n,q)\in\{(5,2),(3,3),(3,5)\}$. In the latter three
groups, no maximal subgroup contains elements of orders~11 and~9, 7 and~12,
respectively~7 and~8.
\par
For $n\le6$ even let $C_1$ contain elements of order $(q^{n-1}+1)/d$. By
\cite[Lemma 2.6]{GM10} the only maximal overgroup of such an element is the
normalizer of $\SU_{n-1}(q)$, or $(n,q)\in\{(4,2),(6,2),(4,3),(4,5)\}$. In
these last four groups, no maximal subgroup contains elements of orders~5 and~9,
11 and~30, 7 and~9, respectively~7 and~13.
\end{proof}

\begin{prop}   \label{prop:genSp}
 Theorem~\ref{thm:strong gen classes} holds for the symplectic groups
 $\PSp_{2n}(q)$, $n\ge2$, $q$ odd for $n\ge3$, $(n,q)\ne(2,2),(2,3)$.
\end{prop}

\begin{proof}
For $n\ge3$ this was already shown in Proposition~\ref{prop:Sp}, using the two
classes in Table~\ref{tab:classtorus}, unless $(n,q)=(4,3)$. For $\PSp_8(3)$
there is only one class of maximal subgroups containing elements of order
$(q^4+1)/2$, viz. the normalizer of an extension field subgroup $\PSp_4(9)$,
by \cite[Lemma 2.8]{GM10}. 
\par
For $n=2$ let $C_1$ consist of elements of order~$(q^2+1)/d$, $d=\gcd(2,q-1)$,
and $C_2$ of regular semisimple elements of order~$q+1$ with centralizer of
order $(q+1)^2/d$. By \cite[Lemma 2.8]{GM10} no maximal subgroup can contain
elements from both classes.
\end{proof}

\begin{prop}   \label{prop:genOodd}
 Theorem~\ref{thm:strong gen classes} holds for the orthogonal groups
 $\OO_{2n+1}(q)$, $n\ge3$.
\end{prop}

\begin{proof}
For $n\ge7$ this was already shown in Proposition~\ref{prop:Oodd}, using the
two classes in Table~\ref{tab:classtorus}. Now assume that $4\le n\le6$. We
take $C_1$ to consist of regular semisimple elements of order $(q^n+1)/d$,
$d=\gcd(2,q-1)$, and $C_2$ containing elements of order $(q^{n-1}+1)(q+1)/d$.
By \cite[Lemma 2.7 and~2.8]{GM10} no maximal subgroup contains both types of
elements, unless possibly when $n=4$, $q=2$, in which case there is no
Zsigmondy prime for $(q^{n-1}+1)/d$. The only maximal subgroups of
$\OO_9(2)=\PSp_8(2)$ of order divisible by~17 are $\OO_8^-(2).2$, $\PSp_4(4).2$
and $\PSL_2(17)$. The latter two do not contain elements of order~20, and
the first only contains one class, while $\OO_9(2)$ contains two such classes.
\par
For $n=3$ let $C_1$ be a class of elements of order $(q^3+1)/d$,
$d=\gcd(2,q-1)$, and $C_2$ a class of elements of order $(q^3-1)/d$. According
to \cite[Lemma~2.7 and~2.8]{GM10} no maximal subgroup contains both types
of elements, unless possibly $q\in\{2,3,4,5\}$. For $\PSp_6(4)=\OO_7(4)$
and $\OO_7(5)$ none of the additional maximal subgroups has elements of
order $(q^3+1)/d$. No maximal subgroup of $\PSp_6(2)=\OO_7(2)$ contains
elements of orders~9 and~15, and no maximal subgroup of $\OO_7(3)$ contains
elements of orders~13 and~14.
\end{proof}

\begin{prop}   \label{prop:genO-}
 Theorem~\ref{thm:strong gen classes} holds for the orthogonal groups
 $\OO_{2n}^-(q)$, $n\ge4$.
\end{prop}

\begin{proof}
This was already shown in Proposition~\ref{prop:O-}, using the two classes in
Table~\ref{tab:classtorus}, except when $(n,q)\in\{(4,2),(5,2),(6,2),(4,4)\}$.
No maximal subgroup of $\OO_8^-(2)$ has order divisible by both~7 and~17;
no maximal subgroup of $\OO_{10}^-(2)$ has order divisible by both~11 and~17;
no maximal subgroup of $\OO_{12}^-(2)$ has elements of order both~11 and~65
by \cite[Lemma~2.10]{GM10}; no maximal subgroup of $\OO_8^-(4)$ has order
divisible by both~257 and~13.
\end{proof}

\begin{prop}   \label{prop:genO+}
 Theorem~\ref{thm:gen classes} holds for the orthogonal groups $\OO_{2n}^+(q)$,
 $n\ge4$.
\end{prop}

\begin{proof}
Let $C_1,C_2$ denote the conjugacy classes of $\OO_{2n}^+(q)$ chosen in the
proof of Proposition~\ref{prop:O+}. Then the claim follows for $n\ge7$ (as
well as for $n=5$ provided that $q\ge5$). For $n=6$ let $C_1,C_2$ consist of
elements with invariant subspaces of types $5^-\oplus1^-$ respectively
$4^-\oplus2^-$. Then any pair of elements from $C_1\times C_2$ acts irreducibly,
and then by Corollary~\ref{cor:twozsig2} they generate $G$. \par
For $n=5$ let $C_1,C_2$ contain elements with invariant subspaces
of types $4^-\oplus1^-$ respectively $3^-\oplus2^-$. Then we conclude as
before unless possibly when $q=2$. No maximal subgroup of $\OO_{10}^+(2)$
has order divisible by~17 and~31. \par
So now assume that $n=4$. Let $C_1, C_2$ contain regular semisimple elements
with invariant subspaces of types $3^-\oplus1^-$ respectively $2^-\oplus2^-$.
Such classes exist whenever $q\ge4$. Let $H\le G$ contain elements from both
classes. Then clearly $H$ is irreducible on the natural module. Thus, by
\cite[Lemma~2.9]{GM10} either $H$ is contained in the normalizer of $\SU_4(q)$,
of $\PSU_3(q)$ or of $\Spin_7(q)$, or $q\in\{2,3,5\}$. The order of
$\PSU_3(q)$ is not divisible by a Zsigmondy prime divisor of $q^2+1$. Regular
semisimple elements of order $q^2+1$ in $\SU_4(q)$ and in $\Spin_7(q)$ have
centralizer order divisible by $q-1$, while elements in $C_2$ have centralizer
order dividing $(q^2+1)^2$. Thus, $H=G$ for $q\notin\{2,3,5\}$. For $q=5$ the
only additional subgroup of order divisible by $13=(q^2+1)/2$ is
$2.\tw2B_2(8)$, but its order is prime to~3. \par

So now assume that $q=3$. Suppose that $G$ is almost simple
with socle $S= \OO^+_8(3)$. Note that $S$ contains $3$ conjugacy
classes of elements of order $20$.  The Sylow $5$-subgroups
of each of these subgroups of order $20$ are not conjugate in $S$.
Thus, any subgroup of $S$ containing elements of order $20$
in more than one class must contain a Sylow $5$-subgroup of $S$.
Thus by \cite{Atl}, the only maximal subgroups containing such elements are
isomorphic to $M:=(\fA_6\times\fA_6).2^2$.   We claim that  $M$ has a unique
conjugacy class of elements of order $20$.  Note that in
$\fA_6 \times \fA_6$, there are $4$ conjugacy classes of elements
of order $20$.  The centralizer of any of them in $M$ is contained
in $\fA_6 \times \fA_6$, whence these classes are fused in $M$.
Note that $M$ embeds in $\fS_6 \wr 2$, whence all elements of order 
$20$ are contained in $\fA_6 \times \fA_6$.
Thus, there are no maximal subgroups of $S$ containing two different
conjugacy classes of elements of order $20$.

If  $|G:S|$ has order
prime to $3$, then there are at least two distinct $G$-classes of elements
of order $20$ in $S$.  The result follows in this case. 
If $3$ does divide $|G:S|$, then  all three classes of elements of order $20$
in $S$ are fused in $G$.   Let $C$ be the set of all elements of order 
$20$ in $S$.  Let $D$ be the $G$-conjugacy class of an outer automorphism
of order a power of $3$.   If $(c,d) \in C \times D$, then by the discussion
above, $S = \langle c , c^d \rangle$, whence the result. 
\par
Finally, let $G=\OO_8^+(2)$. It can be checked by a random computer search
that there are no $\Aut(G)$-invariant subsets $C,D\subset G$ such that any
pair from $C\times D$ generates. Now let $C,D$ be two distinct classes of
elements of order~15 in $G$. These are fused under the triality automorphism.
Note that the third powers of elements from $C$, $D$ are not conjugate either,
so that any subgroup containing elements $x\in C$ and $y\in D$ must have
order divisible by~25. According to the Atlas the only maximal subgroups
with this property are three classes of subgroups $(\fA_5\times\fA_5).2^2$,
(conjugate under triality) and each intersects a unique class of 15-elements
of $G$. Hence $\langle x,y\rangle$ does not lie in any proper subgroup. 
Argue as in the case of $q=3$ to complete the proof. 
\end{proof}

Theorem \ref{thm:strong gen classes} now follows by the propositions
in this section.

A straightforward reduction to the almost simple case now yields
Corollary~\ref{classes}.

%%%%%%%%%%%%%%%%%%%%%%%%%%%%%%%%%%%%%%%%%%%%%%%%%%%%%%%%%%%%%%%%%%%%%%%%%

\end{document}